\documentclass[11pt,a4paper]{amsart}

\usepackage{amssymb,amsmath,graphics,verbatim}
\usepackage{latexsym}
\usepackage{eucal}
\usepackage{a4wide}

\usepackage{color}

\newtheorem{theorem}{Theorem}
\newtheorem{lemma}[theorem]{Lemma}
\newtheorem{prop}[theorem]{Proposition}
\newtheorem{cor}[theorem]{Corollary}
\theoremstyle{remark}

\newcommand{\mc}{\mathcal}
\newcommand{\rr}{\mathbb{R}}
\newcommand{\nn}{\mathbb{N}}
\newcommand{\cc}{\mathbb{C}}

\newcommand{\zz}{\mathbb{Z}}
\newcommand{\cH}{\mathcal{H}}
\newcommand\extunion{\overline{\cup}}

\newcommand{\bX}{\bbar{X}}
\newcommand{\la}{\lambda}
\newcommand{\eps}{\epsilon}

\newcommand{\pl}{\partial}
\newcommand{\x}{\times}

\newcommand{\til}{\widetilde}
\newcommand{\bbar}{\overline}

\newcommand{\cjd}{\rangle}
\newcommand{\cjg}{\langle}

\newcommand{\demi}{\frac{1}{2}}
\newcommand{\ndemi}{\frac{n}{2}}

\newcommand{\lb}{\textrm{lb}}
\newcommand{\rb}{\textrm{rb}}
\newcommand{\ff}{\textrm{ff}}

\newcommand{\cF}{\mathcal{F}}

\newcommand{\dbar}{\overline{\partial}}

\newcommand{\oX}{\overline{X}}

\newcommand{\tr}{\mathrm{tr}}

\newcommand{\cl}{\mathrm{cl}}
\newcommand{\px}{\partial_{x}}
\newcommand{\Ric}{\mathrm{Ric}}
\newcommand{\scal}{\mathrm{scal}}
\newcommand{\tU}{\tilde{U}}

\begin{document}
\title{Bergman and Calder\'on projectors for Dirac operators}
\date{\today}
\author{Colin Guillarmou}
\address{DMA, U.M.R. 8553 CNRS\\
Ecole Normale Sup\'erieure,\\
45 rue d'Ulm\\ 
F 75230 Paris cedex 05 \\France}
\email{cguillar@dma.ens.fr}
\author{Sergiu Moroianu}
\address{Institutul de Matematic\u{a} al Academiei Rom\^{a}ne\\
P.O. Box 1-764\\
RO-014700 Bucharest, Romania}
\email{moroianu@alum.mit.edu}
\author{Jinsung Park}
\address{School of Mathematics\\ Korea Institute for Advanced Study\\
207-43\\ Hoegiro 87\\ Dong\-daemun-gu\\ Seoul 130-722\\
Republic of Korea } \email{jinsung@kias.re.kr}

\begin{abstract}
For a Dirac operator $D_{\bar{g}}$ over a spin compact Riemannian
manifold with boundary $(\bbar{X},\bbar{g})$, we give  a natural
construction of the Calder\'on projector and of the associated Bergman
projector on the space
of harmonic spinors on $\bbar{X}$, and we analyze their Schwartz kernels. Our
approach is based on the conformal covariance of $D_{\bar{g}}$ and the
scattering theory for the Dirac operator associated to the complete conformal metric $g=\bbar{g}/\rho^2$ where
$\rho$ is a smooth function on $\bbar{X}$ which equals the
distance to the boundary near $\pl\bbar{X}$. We show that
$\demi({\rm Id}+\til{S}(0))$ is the orthogonal Calder\'on projector,
where $\til{S}(\la)$ is the holomorphic family in $\{\Re(\la)\geq 0\}$
of normalized scattering operators constructed in \cite{GMP}, which are classical pseudo-differential of order $2\la$.  
Finally we construct natural conformally covariant odd powers of the Dirac operator on any compact spin manifold.
\end{abstract}

\maketitle

\section{Introduction}

Let $(\bbar{X},\bbar{g})$ be a compact spin Riemannian manifold with
boundary, and denote by $(M,h)$ its boundary with the induced spin structure and 
Riemannian metric. Let also $D_{\bar{g}}$ denote the
associated Dirac operator acting on the spinor bundle $\Sigma$ over
$\bX$. The purpose of this paper is to clarify some aspects of the
interaction between the space of smooth spinors of $D_{\bar{g}}$ on
$\bX$ which are harmonic in the interior, and the space of their
restrictions to the boundary. More precisely, we will examine
the orthogonal projectors on these spaces in $L^2$ sense, the
operator of extension from the boundary to a harmonic spinor, and
its adjoint. Before stating our results in the general case, let
us review the situation for the unit disc where one can give
explicit constructions for these objects.
\subsection*{Example: the unit disc}
Keeping the notation $(\bX,\bar{g})$ 
for the closed unit disc in
$\mathbb{C}$ equipped with the Euclidean metric and $M=S^1$, let
\begin{equation*}\label{chcc}
\cH:=\{\phi\in C^\infty(\bX);D_{\bar{g}}\phi=0\}, \qquad
\cH_\pl:=\{\phi|_{M};\phi\in\cH\} \end{equation*} where for the
moment $D_{\bar{g}}=\dbar$ is the Cauchy-Riemann operator. The
functions $z^k,k\geq 0$ clearly are dense in $\cH$ with respect to
the $L^2$ norm. Their restrictions to the boundary $e^{ikt}$,
$k\geq 0$, span the space of those smooth functions whose Fourier
coefficients corresponding to negative frequencies vanish. The
orthogonal projection $P_{\cH_{\pl}}$ onto the $L^2$-closure of
$\cH_\pl$ is easily seen to be pseudodifferential; if $A={-}id/dt$
is the self-adjoint Dirac operator on $M$, then $P_{\cH_{\pl}}$ is
the Atiyah-Patodi-Singer projection on the {non-negative} part of
the spectrum of $A$, whose kernel is given by $(2\pi(1-z\bar{w}))^{-1}$ 
with respect to the measure $dt$ where $w=e^{it}$. 
Let $K:C^\infty(M)\to C^\infty(\bX)$ be the
operator which to $\phi_{|M}\in \cH_\pl$ associates $\phi$,
extended by $0$ on the orthogonal complement of $\cH_\pl$. Then
$K$ has a smooth kernel on ${X}\times M$ where
$X=\mathrm{int}(\bbar{X})$ given by
\begin{equation}\label{form1}
K(z,w)= \frac{1}{2\pi(1-z\bbar{w})}
\end{equation}
with respect to the standard measure on the circle, where $w=e^{it}$.
This kernel extends to $\bX\times M$ with a singularity
at the boundary diagonal $\{(z,w);z=w\}$. If we set
\begin{align*}z=(1-x)e^{i(t+y)},&&\rho:=\sqrt{x^2+y^2}\end{align*}
we see that the leading term of the singularity is $\rho^{-1}$,
moreover $K(z,w)$ admits a power series expansion near $\rho=0$.
The coefficients live on the ``polar coordinates'', or blow-up
space which will play an essential role in the rest of this
paper. The adjoint of $K$, denoted by $K^*$, has a smooth kernel on
$M\times X$ with respect to the standard measure $\frac{1}{2
i}dz\wedge d\bbar{z}$, given formally by \eqref{form1}. This has
the same type of singularity as $K$ near $\{z=w\}$.  The  kernel of $K^*K$ on $M\times M$ is given by
\[-\frac{1}{4\pi}\frac{\log(1-z\bbar{w})}{z\bbar{w}},\]
which is the kernel of a classical pseudo-differential operator of order $-1$ (actually given by $\demi P_{\cH_{\pl}}(|D_t|+1)^{-1}$).
The remaining composition $KK^*$ has a smooth kernel
on $X\times X$ given by $(2\pi (1-z\bbar{w}))^{-1}$ with respect to the Euclidean measure in $w$. 
This kernel extends
to $\bX\times \bX$ with the same type of singularity as in the
case of $K$ and $K^*$, only that now the singular locus is of
codimension $3$, and there are two, instead of one, extra boundary
hyperfaces. To finish our example, consider the projector on (the
closure of) $\cH$. Its kernel with respect to $\frac{1}{2i}dw\wedge d\bbar{w}$ is
\[\frac{1}{\pi(1-z\bbar{w})^2}\]
which is of the same nature as the kernel of $KK^*$ but with a
higher order singularity.

\subsection*{Harmonic spinors on manifolds with boundary} 
One can extend the above example to higher complex dimensions.
One direction would be to study holomorphic functions smooth
up to the boundary, however in this paper we will consider another
generalization. Let thus $\oX$ be a compact domain in $\cc^n$, and
$D_{\bar{g}}=\dbar+\dbar^*$ acting on $\Lambda^{0,*}X$. A form is called
\emph{harmonic} if it belongs to the nullspace of $D_{\bar{g}}$. Then the
above analysis of the operators $K,K^*$ and of the projection on the space
of harmonic forms can be carried out, describing the singularities of the
kernels involved. In fact, even
more generally, we will consider the Dirac operator $D_{\bar{g}}$
acting on the spinor
bundle $\Sigma$ over a compact spin manifold $\bbar{X}$ with
boundary $M$. We assume that the metric $\bbar{g}$ on $\bX$ is
smooth at the boundary but not necessarily of product type (which
would mean that the gradient of the distance function $\rho$ to
the boundary were Killing near the boundary). We then denote by
$\mc{H}(D_{\bar{g}})$ and $\cH_\pl(D_{\bar{g}})$ the space of smooth
harmonic spinors and  the Cauchy data space of $D_{\bar{g}}$
respectively,
\begin{align*}
\mc{H}(D_{\bar{g}}):=\{\phi\in C^{\infty}(\bbar{X};\Sigma);
D_{\bar{g}}\phi=0\},&& \mc{H}_\pl(D_{\bar{g}}):=\{\phi|_{M};
\phi\in \mc{H}(D_{\bar{g}})\}
\end{align*}
and let $\bbar{\mc{H}}(D_{\bar{g}})$ and
$\bbar{\mc{H}}_{\pl}(D_{\bar{g}})$ be their respective $L^2$ closures.
When the dependence on $D_{\bar{g}}$ is clear, we may omit $D_{\bar{g}}$
in the notations $\cH(D_{\bar{g}})$, $\cH_{\pl}(D_{\bar{g}})$ for
simplicity. We denote  by $P_{\bbar{\mc{H}}}$ and
$P_{\bbar{\mc{H}}_\pl}$ the respective orthogonal projectors on
$\bbar{\mc{H}}$ (that we call the \emph{Bergman projector}) and $\bbar{\mc{H}}_\pl$ 
(the \emph{Calder\'on projector}) for the $L^2$ inner
product induced by $\bbar{g}$ and $\bbar{g}|_{M}$. Let $K:L^2(M,\Sigma)\to
L^2(\bbar{X},\Sigma)$ be the \emph{Poisson operator}, i.e., the extension map which sends
$\bbar{\mc{H}}_\pl$ to $\bbar{\mc{H}}$, that is, $D_{\bar{g}}K\psi=0$ and $K\psi|_{\pl\bbar{X}}=\psi$
for all $\psi\in \mc{H}_{\pl}$, and denote by
$K^*:L^2(\bX,\Sigma)\to L^2(M,\Sigma)$ its adjoint. The main results in this paper
concern the structure of the Schwartz kernels of these operators,
which also gives new proofs for some known results.

Let us remark that the construction of the orthogonal projector
$P_{\bbar{\mc{H}}_\pl}$ called here \emph{Calder\'on projector},
and its applications, have been a central subject in the global
analysis of manifolds with boundary since the works of Calder\'on
\cite{Calderon} and Seeley \cite{Seeley66}, \cite{Seeley69}. The
Calder\'on projector of Dirac-type operators turned out to play a
fundamental role in geometric problems related to
analytic-spectral invariants. This was first
observed by Bojarski in the linear conjugation problem of the
index of a Dirac type operator \cite{Bor}. Following Bojarski,
Booss and Wojciechowski extensively studied the geometric aspect
of the Calder\'on projector \cite{BoW}. The Calder\'on projector
also appears in the gluing formulae of the analytic-spectral
invariants studied in \cite{Nic}, \cite{SW}, \cite{KL}, \cite{LP}
since the use of the Calder\'on projector provides us with more
refined proofs of these formulae in more general settings. We also
refer to \cite{Epstein1}, \cite{Epstein2} for an application of
the Calder\'on projector of the ${\rm Spin}\sb {\mathbb C}$ Dirac
operator, and a recent paper of Booss-Lesch-Zhu \cite{BLZ} for
other generalizations of the work in \cite{BoW}. Extensions of the
Calder\'on projector for non-smooth boundaries were studied
recently in \cite{AmNis,loya}.

\subsection*{Polyhomogeneity} 
Before we state the main results of this paper, let us fix a couple of notations and
definitions. If $W$ and $Y$ are smooth compact manifolds (with or
without boundary) such that the corner of highest codimension of $W\x
Y$ is diffeomorphic to a product $M\x M$ where $M$ is a closed
manifold, we will denote, following Mazzeo-Melrose
\cite{MM}, by $W\x_0Y$ the smooth compact manifold with corners
obtained by blowing-up the diagonal $\Delta$ of $M\x M$ in $W\x
Y$, i.e., the manifold obtained by replacing the submanifold
$\Delta$ by its interior pointing normal bundle in $W\x Y$ and
endowed with the smallest smooth structure containing the lift of
smooth functions on $W\x Y$ and polar coordinates around $\Delta$.
The bundle replacing the diagonal creates a new boundary
hypersurface which we call the \emph{front face} and we denote by $\ff$.
A smooth boundary defining function of $\ff$ in $W\x_0Y$ can be
locally taken to be the lift of $d(\cdot,\Delta)$, the Riemannian
distance to the submanifold $\Delta$. On a smooth compact manifold
with corners $W$, we say that a function (or distribution) has an
\emph{integral polyhomogeneous expansion} at the boundary
hypersurface $H$ if it has an asymptotic expansion at $H$ of the form
\begin{equation}\label{intphgexp}
\sum_{j=-J}^\infty\sum_{\ell=0}^{\alpha(j)} q_{j,\ell}\,
\rho_{H}^{j}(\log\rho_H)^{\ell}
\end{equation} 
for some
$J\in\nn_0{:=\{0\}\cup\nn}$, a non-decreasing function
$\alpha:\zz\to \nn_0$, and some smooth functions
$q_{j,\ell}$ on ${\rm int}(H)$, where $\rho_H$ denotes any smooth
boundary defining function of $H$ in $W$.

\begin{theorem}\label{th1}
Let $(\bbar{X},\bbar{g})$ be a smooth compact spin Riemannian manifold with boundary $M$. Let $K$ be the Poisson operator
for $D_{\bar{g}}$ and let $K^*$ be its adjoint.  Then the following hold true:
\begin{enumerate}
\item The Schwartz kernels of $K$, $K^*$ and $KK^*$ are smooth on the blown-up spaces
$\bbar{X}\x_0 M$,  $M\x_0\bbar{X}$, respectively $\bbar{X}\x_0\bbar{X}$ with respect to the volume densities 
induced by $\bbar{g}$.
\item{The operator $K^*K$ is a classical pseudo-differential
operator of order $-1$ on $M$ which maps $L^2(M,\Sigma)$ to
$H^1(M,\Sigma)$, and there exists  a pseudo-differential operator
of order $1$ on $M$ denoted by $(K^*K)^{-1}$ such that the Calder\'on projector 
$P_{\bbar{\mc{H}}_{\pl}}$ is given by $(K^*K)^{-1}K^*K$. In particular, $P_{\bbar{\mc{H}}_{\pl}}$ is classical pseudo-differential of order $0$.}
\item{The Bergman orthogonal projection $P_{\bbar{\mc{H}}}$ from
$L^2(\bbar{X},\Sigma)$ to $\bbar{\mc{H}}$ is given by
$K(K^*K)^{-1}K^*$ and its Schwartz kernel on $\bbar{X}\x_0\bbar{X}$ is smooth except at the front face ${\rm ff}$ where it has integral polyhomogeneous expansion as in \eqref{intphgexp} with $\alpha\leq 3$.}
\end{enumerate}
\end{theorem}
Note that an alternate description of these kernels in terms of oscillatory integrals is given in Appendix \ref{appB}.

Our method of proof is to go through an explicit
construction of all these operators which does not seem to be
written down in the literature in this generality for the Dirac
operator, although certainly some particular aspects are well known (especially
those involving the Calder\'on projector $P_{\bbar{\mc{H}}_\pl}$,
see \cite{BoW}). We use the fundamental property that the Dirac
operator is conformally covariant to transform the problem into a problem on a complete non-compact
manifold $(X,g)$ conformal to $(\bbar{X},\bbar{g})$ obtained by
simply considering $X:={\rm int}(\bbar{X})$ and $g:=\bbar{g}/\rho^2$
where $\rho$ is a smooth boundary defining function of the
boundary $M=\pl\bbar{X}$ which is equal to the distance to the
boundary (for the metric $\bbar{g}$) near $\pl\bbar{X}$. This kind
of idea is not really new since this is also in spirit  used for
instance to study pseudoconvex domains by considering a complete
K\"ahler metric in the interior of the domain (see
Donnelly-Fefferman \cite{Donnelly-Fefferman}, Fefferman \cite{Fe},
Cheng-Yau \cite{Cheng-Yau}, Epstein-Melrose \cite{Ep-Mel}), and
obviously this connection is transparent for the disc in $\cc$ via
the Poisson kernel and the relations with the hyperbolic plane.
One of the merits of this method, for instance, is that we do not
need to go through the invertible double of \cite{BoW} to
construct the Calder\'on projector and thus we do not need the
product structure of the metric near the boundary. We finally remark 
that the bound $\alpha\leq 3$ in (3) of the Theorem is almost certainly 
not optimal, we expect instead $\alpha\leq 1$ to be true. 

\subsection*{Conformally covariant operators}
We also obtain, building on our previous work
\cite{GMP},
\begin{theorem}\label{th2}
There exists a holomorphic family in $\{\la\in \cc ; \Re(\la)\geq
0\}$ of elliptic pseudo-differential operators $\til{S}(\la)$ on
$M=\pl\bbar{X}$ of complex order $2\la$, invertible except at a
discrete set of $\la$'s and with principal symbol $i\cl(\nu)\cl(\xi)|\xi|^{2\la-1}$
where $\nu$ is the inner unit normal vector field to $M$ with respect to $\bar{g}$, such that
\begin{enumerate}
\item [(a)]{$\demi({\rm Id}+\til{S}(0))$ is the Calder\'on projector
$P_{\bbar{\mc{H}}_\pl}$;} 
\item[(b)]{For $k\in\nn_0$, $L_k:=-\cl(\nu)\til{S}(1/2+k)$ is a conformally covariant
differential operator whose leading term is $D_M^{1+2k}$ where $D_M$ denotes the Dirac operator on $M$,
and $L_0=D_M$.}
 \end{enumerate}
\end{theorem}

By using the existence of ambient (or Poincar\'e-Einstein) metric of Fefferman-Graham 
\cite{FGR,FGR2}, this leads to the construction of natural conformally covariant powers 
of Dirac operators in degree $2k+1$ on any spin Riemannian manifolds $(M,h)$ of dimension 
$n$, for all $k\in \nn_0$ if $n$ is odd and for $k\leq n/2$ if $n$ is even. We explicitly compute $L_1$. 
\begin{cor} \label{corL1}
Let $(M,h)$ be a Riemannian manifold of dimension $n\geq 3$ with a fixed 
spin structure, and denote by $\scal,\Ric$ and $D$ the scalar curvature, 
the Ricci curvature, and respectively the Dirac operator with respect to $h$. 
Then the operator $L_1$ defined by 
\[ L_1:=D^3 -\frac{\cl(d(\scal))}{2(n-1)}-\frac{2\,\cl\circ\Ric\circ\nabla}{n-2}
+\frac{\scal}{(n-1)(n-2)}D
\]
is a natural conformally covariant differential operator:
\[\hat{L}_1= e^{-\frac{n+3}{2}\omega}L_1e^{\frac{n-3}{2}} \]
if $\hat{L}_1$ is defined in terms of the conformal metric $\hat{h}=e^{2\omega}h$.
\end{cor}

\subsection*{Cobordism invariance of the index and local Wodzicki-Guillemin 
residue for the Calder\'on projector} 
As a consequence of Theorem \ref{th2} and the analysis of \cite{GMP}, we deduce the following
\begin{cor}
Let $(\bbar{X},\bbar{g})$ be a smooth compact spin Riemannian manifold with boundary $M$.
\begin{enumerate}
\item The Schwartz kernel of the Calder\'on projector $P_{\bbar{\mc{H}}_\pl}$ associated to the Dirac operator 
has an asymptotic expansion in polar coordinates 
around the diagonal without log terms. In particular, the Wodzicki-Guillemin local residue density of $P_{\bbar{\mc{H}}_\pl}$  vanishes.
\item When the dimension of $M$ is even, the spinor bundle
$\Sigma$ splits in a direct sum $\Sigma_+\oplus\Sigma_-$. 
If $D^+_M$ denotes $D_M|_{\Sigma_+}:\Sigma_+\to \Sigma_-$, then the index ${\rm Ind}(D^+_M)$ is $0$.
\end{enumerate}
\end{cor}

As far as we know, the first part of the corollary is new. It is known since Wodzicki  \cite{Wod} that the \emph{global} residue trace 
of a pseudo-differential projector of order $0$ vanishes, however the local residue density does not vanish for general projectors 
(e.g.\ see \cite{Gil}). What is true is that the APS spectral projector has also vanishing local residue, a fact which is equivalent to the conformal invariance
of the eta invariant. For metrics of product type near the boundary the Calder\'on and APS projectors coincide up to smoothing operators; thus our result was known for such metrics.

The second statement is the well-known cobordism invariance of the index for the Dirac operator;  there exist several proofs
of this fact for more general Dirac type operators (see for instance \cite{AS, Moro,Lesch,Nico, Brav}) but we found it worthwhile to
point out that this fact can be obtained as a easy consequence of the invertibility of the scattering operator. In fact, a proof of cobordism invariance using scattering theory for cylindrical metrics has been found recently by M\"uller-Strohmaier \cite{MuSt}, however their approach does not seem to have implications about the Cald\'eron or Bergman projectors.

\subsection*{More general operators} Our approach does not seem to work for more general Dirac type operators. However 
it applies essentially without  modifications to twisted spin Dirac operators, with twisting bundle and connection smooth on $\bbar{X}$. For simplicity of notation, we restrict ourselves to the untwisted case.

\subsection*{Acknowledgements} 
This project was started while the first two authors were visiting KIAS Seoul,
it was continued while C.G. was visiting IAS Princeton, and finished while S.M. was visiting ENS Paris;
we thank these institutions for their support. We also thank Andrei Moroianu for checking (with an independent  method)
the formula for $L_1$ in Corollary \ref{corL1}. 
 C.G. was supported by the grant NSF-0635607 at IAS.
S.M. was supported by the grant PN-II-ID-PCE 1188 265/2009 and by a CNRS grant at ENS. 

\section{Dirac operator on asymptotically hyperbolic manifold}\label{AH}
We start by recalling the results of \cite{GMP} that we need for
our purpose. Let $(X,g)$ be an $(n+1)$-dimensional smooth complete
non-compact spin manifold which is the interior of a smooth compact manifold with boundary 
$\bbar{X}$. We shall say that it is \emph{asymptotically
hyperbolic} if the metric $g$ has the following properties: there
exists a smooth boundary defining function $x$ of $\pl\bbar{X}$
such that $x^2g$ is a smooth metric on $\bbar{X}$ and
$|dx|_{x^2g}=1$ at $\pl\bbar{X}$. It is shown in \cite{GRL,JSB}
that for such metrics, there is a diffeomorphism
$\psi:[0,\eps)_t\x \pl\bbar{X}\to U\subset\bbar{X}$ such that
\begin{equation}\label{psig}
\psi^*g=\frac{dt^2+h(t)}{t^2}
\end{equation}
where $\eps>0$ is small, $U$ is an open neighborhood of
$\pl\bbar{X}$ in $\bbar{X}$ and $h(t)$ is a smooth one-parameter
family of metrics on $\pl\bbar{X}$.  The function $\psi_*(t)$ will
be called \emph{geodesic boundary defining function} of
$\pl\bbar{X}$ and the metric $g$ will be said \emph{even to order
$2k+1$} if $\pl^{2j+1}_th(0)=0$ for all $j<k$; such a property
does not depend on $\psi$, as it is shown in \cite{GuiDMJ}. The
\emph{conformal infinity} of $\bbar{X}$ is the conformal class on
$\pl\bbar{X}$ given by
\[[h_0]:=\{(x^2g)|_{T\pl\bbar{X}}\ ; \ x\textrm{ is a boundary defining function of }\bbar{X}\}.\]
On $\bbar{X}$ there exists a natural smooth bundle ${^0T}\bbar{X}$
whose space of smooth sections is canonically identified with  the Lie
algebra $\mc{V}_0$ of smooth vector fields which vanish at the
boundary $\pl\bbar{X}$, its dual ${^0T}^*\bbar{X}$ is also a smooth
bundle over $\bbar{X}$ and $g$ is a smooth metric on
${^0T}\bbar{X}$.

Consider the ${\rm SO}(n+1)$-principal bundle ${^0_o}F(\bbar{X})\to
\bbar{X}$ over $\bbar{X}$ of orthonormal frames in ${^0T}\bbar{X}$
with respect to $g$. Since $\bbar{X}$ is spin, there is a ${\rm
Spin}(n+1)$-principal bundle ${^0_s}F(\bbar{X})\to \bbar{X}$ which
double covers ${^0_o}F(\bbar{X})$ and is compatible with it in the
usual sense. The $0$-Spinor bundle $^0\Sigma(\bbar{X})$ can then be
defined as a bundle associated to the ${\rm Spin}(n+1)$-principal
bundle ${^0_s}F(\bbar{X})$, with the fiber at $p\in\bbar{X}$
\[^0\Sigma_p(\bbar{X})=({^0_s}F_p\x S(n+1))/\tau\]
where $\tau:{\rm Spin}(n+1)\to {\rm Hom}(S(n+1))$ is the standard spin
representation on $S(n+1)\simeq \cc^{2^{[(n+1)/2]}}$. If $x$ is
any geodesic boundary defining function, the unit vector field
$x\pl_x:=\nabla^{g}\log(x)$ is a smooth section of ${^0T}\bbar{X}$. The Clifford multiplication ${\rm cl}(x\pl_x)$ restricts to
the boundary to a map denoted by ${\rm cl}(\nu)$, independent of the choice of $x$, 
satisfying ${\rm cl}(\nu)^2=-{\rm Id}$ which splits the
space of 0-spinors on the boundary into $\pm i$ eigenspaces
\begin{align*}{^0\Sigma_\pm}:=\ker ({\rm cl}(\nu)\mp i),&& {^0\Sigma}|_{M}={^0\Sigma}_+\oplus {^0\Sigma}_-\end{align*}

The Dirac operator $D_g$ associated to $g$ acts in
$L^2(X,{^0\Sigma})$ and is self-adjoint since the
metric $g$ is complete. Let us denote by
$\dot{C}^\infty(\bbar{X},{^0\Sigma})$ the set of smooth spinors on
$\bbar{X}$ which vanish to infinite order at $\pl\bbar{X}$. We
proved the following result in  \cite[Prop 3.2]{GMP}:
\begin{prop}\label{resolvent}
The spectrum of $D_g$ is absolutely continuous and given by the
whole real line $\sigma(D_g)=\rr$. Moreover the $L^2$ bounded
resolvent $R_\pm(\la):=(D_g\pm i\la)^{-1}$ extends from $\{\Re(\la)>0\}$
meromorphically in $\la\in\cc\setminus {-\nn/2}$ as a family of
operators mapping $\dot{C}^\infty(\bbar{X},{^0\Sigma})$ to
$x^{\ndemi+\la}C^\infty(\bbar{X},{^0\Sigma})$, and it is analytic
in $\{\Re(\la)\geq 0\}$. Finally, we have
$[x^{-\ndemi-\la}R_\pm(\la)\sigma]|_{\pl\bbar{X}}\in
C^{\infty}(\pl\bbar{X},{^0\Sigma}_\mp)$ for all
$\sigma\in\dot{C}^\infty(\bbar{X},{^0\Sigma})$.
\end{prop}

Using this result, in \cite{GMP} we were able  to solve the
following boundary value problem
\begin{prop}\label{poisson}
Let $\la\in U:=\{z\in\cc;\Re(z)\geq 0, z\notin \nn/2\}$. For
all $\psi\in C^{\infty}(\pl\bbar{X},{^0\Sigma}_\pm)$ there is a
unique $\sigma_\pm(\la)\in C^{\infty}(X,{^0\Sigma})$ such that
there exist $\sigma_\pm^+(\la),\sigma_{\pm}^-(\la)\in
C^{\infty}(\bbar{X},{^0\Sigma})$ satisfying
$\sigma_\pm(\la)=x^{\ndemi-\la}\sigma_\pm^-(\la)+x^{\ndemi+\la}\sigma_\pm^+(\la)$
and 
\begin{align}\label{eq-sigma} 
(D_g\pm i\la)\sigma_\pm(\la)=0, && \sigma^-_\pm(\la)|_{\pl\bbar{X}}=\psi.
\end{align} 
Moreover
$\sigma_\pm^+(\la),\sigma_\pm^-(\la)$ are analytic in $\la\in U$
and one has $\sigma_\pm^+(\la)|_{\pl\bbar{X}}\in
C^\infty(\pl\bbar{X},{^0\Sigma}_\mp)$.
\end{prop}
The solution $\sigma_{\pm}(\la)$ of Proposition
\ref{poisson} is constructed in Lemma 4.4 of \cite{GMP} as a sum 
\begin{equation}\label{constsigma}
\sigma_\pm(\la)=\sigma_{\infty,\pm}(\la)-R_\pm(\la)(D_g\pm i\la)\sigma_{\infty,\pm}(\la)
\end{equation}
where $\sigma_{\infty,\pm}(\la)\in
x^{\ndemi-\la}C^{\infty}(\bbar{X},{^0\Sigma})$ satisfies
\begin{align}\label{sigmainfty}
(D_g\pm i\la)\sigma_{\infty,\pm}(\la)\in
\dot{C}^\infty(\bbar{X},{^0\Sigma}), &&
[x^{-\ndemi+\la}\sigma_{\infty,\pm}(\la)]|_{\pl\bbar{X}}=\psi\end{align}
with the additional property that it is analytic in
$\{\Re(\la)\geq 0, \la\notin \nn/2\}$. Since $R_\pm(\la)$ are
analytic in $\{\Re(\la)\geq 0\}$, this shows that $\sigma_\pm(\la)$ is
analytic in the same domain, and we have $D_g\sigma_{\pm}(0)=0$.
Since this will be useful below, we recall briefly the
construction  of the approximate solution
$\sigma_{\infty,\pm}(\la)$ near the boundary from Lemma 4.4 in
\cite{GMP}. The principle is to write the Dirac operator near
$\pl\bbar{X}$ in the product decomposition $[0,\eps)_x\x\pl\bbar{X}$
\begin{equation}\label{diracAH}
D_g=x^{\ndemi}({\rm cl}(x\pl_x)x\pl_x+xD_{h_0})x^{-\ndemi}+ xP
\end{equation}
where $D_{h_0}$ is the Dirac operator on the boundary for the metric $h_0$ and $P$ is
a first order differential operator with smooth coefficients which in local
coordinates $(x,y)$ near the boundary can be written
\[P=P_0(x,y)x\pl_x +\sum_{j=1}^nP_j(x,y)x\pl_{y_i}\]
for some smooth sections $P_j$ of ${^0\Sigma}\otimes
{^0\Sigma}^*$. Consequently, one has for any $\psi_\pm\in
C^{\infty}(\pl\bbar{X},{^0\Sigma}_{\pm})$ and $k\in\nn_0$ the
indicial equation
\begin{align}\label{indicialeq}
(D_g\pm i\la)&x^{\ndemi-\la+k}(\psi_++\psi_-)\\
=&ix^{\ndemi-\la+k}\Big((k-\la\pm\la)\psi_++(\la-k\pm\la)\psi_-\Big)+
x^{\ndemi-\la+k+1}F_\la^k\notag
\end{align} 
where $F^k_\la\in C^{\infty}(\bbar{X},{^0\Sigma})$ is holomorphic
near $\la=0$. From this, using formal series and Borel lemma, it
is easy to see that one can construct near $\la=0$ a spinor
$\sigma_{\infty,\pm}(\la)\in
x^{\ndemi-\la}C^{\infty}(\bbar{X},{^0\Sigma})$, holomorphic near
$\la=0$, solving \eqref{sigmainfty} whose formal Taylor series is
determined locally and uniquely by $\psi_\pm$.

Let $\sigma_{\pm}(\la)$ be the spinor of Proposition
\ref{poisson} (thus depending on $\psi$), we can then define
linear Poisson operators and scattering operators
\begin{align*}
E_\pm(\la):C^{\infty}(\pl\bbar{X},{^0\Sigma}_\pm)\to &C^{\infty}(X,{^0\Sigma}), & \psi  \mapsto&  \sigma_{\pm}(\la), \\
\ \ \, S_\pm (\la):
C^{\infty}(\pl\bbar{X},{^0\Sigma}_\pm)\to  &C^{\infty}(\pl\bbar{X},{^0\Sigma}_\mp), &
\psi  \mapsto & \sigma^+_{\pm}(\la)|_{\pl\bbar{X}}
\end{align*}
which are holomorphic in $\{\Re(\la)\geq 0, \la\notin \nn/2\}$. We extend the definition of $E_\pm(\la)$
to the whole bundle ${^0\Sigma}$ by setting that it acts by $0$ on ${^0\Sigma}_\mp$.
Then from Proposition 4.6 of \cite{GMP}, the Schwartz kernel $E_\pm(\la;m,y')\in C^{\infty}(X\x \pl\bbar{X};
{^0\Sigma}\otimes{^0\Sigma}^*)$ of $E_\pm(\la)$
is given by
\begin{equation}\label{kernelE}
E_\pm(\la;m,y')=[R_\pm(\la;m,x',y'){x'}^{-\ndemi-\la}]|_{x'=0}{\rm cl}(\nu)
\end{equation}
where $R_\pm(\la;m,m')$ is the Schwartz kernel of $R_\pm(\la)$.
We can also define
\begin{align}\label{defelasla}
\, E(\la):C^{\infty}(\pl\bbar{X},{^0\Sigma})\to &C^{\infty}(X,{^0\Sigma}),&
\psi_++\psi_-\mapsto & E_+(\la)\psi_++E_-(\la)\psi_-, \\
S(\la):C^{\infty}(\pl\bbar{X},{^0\Sigma})\to &C^{\infty}(\pl\bbar{X},{^0\Sigma})\,&
\psi_++\psi_-\mapsto & S_+(\la)\psi_++S_-(\la)\psi_- .\nonumber
\end{align}
The main features of $S(\la)$, also proved in Section 4.3 of
\cite{GMP}, are gathered in
\begin{prop}\label{propofS}
For $\Re(\la)\geq 0$ and $\la\notin \nn/2$, the operator $S(\la)$ depends 
on the choice of the boundary defining function $x$ but changes under the law
\begin{align}\label{change}
\hat{S}(\la)=e^{-(\ndemi+\la)\omega_0}S(\la)e^{(\ndemi-\la)\omega_0}, && \omega_0:=\omega|_{x=0}
\end{align}
if $\hat{S}(\la)$ is the scattering operator defined using the boundary defining function $\hat{x}=e^{\omega}x$
for some $\omega\in C^{\infty}(\bbar{X})$.
Moreover $S(\la)\in \Psi^{2\la}(\pl\bbar{X},{^0\Sigma})$ is a classical pseudodifferential operator of order $2\la$, and its principal symbol is given by
\[\sigma_{\rm pr}(S(\la))(\xi)=i2^{-2\la}\frac{\Gamma(1/2-\la)}{\Gamma(1/2+\la)}{\rm cl }(\nu){\rm cl}(\xi)|\xi|^{2\la-1}_{h_0}\]
where $h_0=(x^2g)|_{T\pl\bbar{X}}$. If $\la\in i\rr$, $S(\la)$ extends as a
unitary operator on $L^2(\pl\bbar{X},{^0\Sigma})$,
its inverse is given by $S(-\la)$ and extends meromorphically in $\{\Re(\la)\geq 0,\la\notin \nn/2\}$
as a family of classical pseudo-differential operators in $\Psi^{-2\la}(\pl\bbar{X},{^0\Sigma})$. Finally $S(\la)$
is self-adjoint for $\la\in (0,\infty)$.
\end{prop}
The conformal change law and the invertibility are easy
consequences of the definition of $S(\la)$ and the uniqueness of
the solution $\sigma_{\pm}(\la)$ in Proposition \ref{poisson}, the
pseudodifferential properties and the meromorphic extension are
more delicate and studied in Section 4.3 of \cite{GMP}. In
particular, by letting $\la\to 0$ in \eqref{eq-sigma}, we deduce
easily the following
\begin{prop}\label{p:harmonic}
Let $\psi\in C^\infty(\pl\bbar{X},{^0\Sigma})$, then
$\sigma:=E(0)\psi$ is a harmonic spinor for $D$, which lives in
$x^\ndemi C^\infty(\bbar{X},{^0\Sigma})$ and has the following
behavior at the boundary
\[\sigma= x^{\ndemi}({\rm Id}+S(0))\psi +O(x^{\ndemi+1}).\]
\end{prop}
Remark from Proposition \ref{propofS} that $S(0)^*=S(0)^{-1}=S(0)$
and so the operator
\begin{equation}\label{defc}
\mc{C}:=\demi({\rm Id}+S(0))
\end{equation}
is an orthogonal projector on a subspace of
$L^2(\pl\bbar{X},{^0\Sigma})$ for the measure ${\rm dv}_{h_0}$
where $h_0=(x^{2}g)|_{T\pl\bbar{X}}$. Notice from \eqref{change}
that, under a change of boundary defining function
$\hat{x}=e^{\omega}x$, the operator $\mc{C}$ changes according to
conjugation $\hat{\mc{C}}=e^{-\ndemi\omega_0}\mc{C}
e^{\ndemi\omega_0}$.

Now we want to prove that the range of $E(0)$ acting on
$C^{\infty}(\pl\bbar{X},{^0\Sigma})$ is exactly the set of
 harmonic spinors in
$x^{\ndemi}C^{\infty}(\bbar{X},{^0\Sigma})$.
\begin{prop}\label{span}
Let $\phi\in x^{\ndemi}C^{\infty}(\bbar{X},{^0\Sigma})$ such that
$D_g\phi=0$ and let $\psi:=(x^{-\frac n2}\phi)|_{\pl\bbar{X}}$. Then
we have $E(0)\psi=2\phi$.
\end{prop}
\begin{proof} First let us write $\psi=\psi_++\psi_-$ with $\psi_\pm\in{^0\Sigma}_\pm$. Then
we construct the approximate solution $\sigma_{\infty,+}(\la)$ of
\eqref{sigmainfty} associated to $\psi_+$. Let us set
$\phi_+(\la):=\sigma_{\infty,+}(\la)$ and
$\phi_-(\la):=\phi-\phi_+(\la)$. One has
$(x^{-\ndemi}\phi_-({0}))|_{x=0}=\psi_-\in{^0\Sigma}_-$ and
$D_g\phi_-(0)=-D_g\phi_+(0)$. As in the proof of Proposition
\ref{poisson}, we have
\[\sigma_{+}(\la)=\phi_+(\la)-R_+(\la)(D_g+ i\la)\phi_+(\la)=E_+(\la)\psi_+.\]
and in particular, since all the terms in the composition on the
right hand side are holomorphic near $\la=0$, we obtain that
\[E_+(0)\psi_+=\phi_+(0)-R_+(0)D_g\phi_+(0)=\phi_+(0)+R_+(0)D_g\phi_-(0).\]
Now we use Green's formula on a region $\{x\leq \eps\}$ for $\eps>0$ small and by letting
$\eps\to 0$ we deduce easily from \eqref{kernelE} that
\[R_+(0)D_g\phi_-(0)=\phi_-(0)-E_+(0)\psi_-=\phi_-(0).\]
Consequently, we have
proved that $E_+(0)\psi_+=\phi_+(0)+\phi_-(0)=\phi$. A similar
reasoning shows that $E_-(0)\psi_-=\phi$ and this achieves the
proof.
\end{proof}

As a corollary we deduce that $S(0)\psi=\psi$ for $\psi$ as in Proposition \ref{span}, so
\begin{theorem}\label{projector}
The following identity holds for $\mc{C}=\demi({\rm Id}+S(0))$
\[\{(x^{-\ndemi}\sigma)|_{\pl\bbar{X}}; \sigma\in x^{\ndemi}C^{\infty}(\bbar{X},{^0\Sigma}), D_g\sigma=0\}=
\{\mc{C}\psi;\psi\in C^{\infty}(\pl\bbar{X},{^0\Sigma})\}.\]
\end{theorem}

\section{Dirac operator on compact manifolds with boundary}

\subsection{Calder\'on projector and scattering operator at $0$}
Now we let $D_{\bar{g}}$ be the Dirac operator on a smooth compact
spin manifold with boundary $(\bbar{X},\bbar{g})$, and we denote by $\Sigma$ the
spinor bundle. We recall that the \emph{Cauchy data space} of
$D_{\bar{g}}$ is given by
\[\mc{H}_\pl:=\{\phi|_{\pl\bbar{X}}, \phi\in C^{\infty}(\bbar{X},\Sigma), D_{\bar{g}}\phi=0\}\]
i.e., it is the space of boundary values of smooth harmonic spinors on
$\bbar{X}$ for $D_{\bar{g}}$. The orthogonal \emph{Calder\'on
projector} $P_{\bbar{\cH}_\pl}$ is a projector acting on
$L^2(\pl\bbar{X},\Sigma)$ and whose range is the $L^2$-closure
$\bbar{\mc{H}}_\pl$. Booss and Wojciechowski \cite{BoW} studied
Fredholm properties of boundary value problems for Dirac type
operators on manifolds with boundary, they found that if $P$ is a
pseudo-differential projector on the boundary, the operator 
$D_P^+:{\rm
Dom}(D_P^+)\to C^{\infty}(\bbar{X},{\Sigma}^+)$ with domain
\[{\rm Dom}(D_P^+):=\{\phi\in C^{\infty}(\bbar{X},\Sigma^+); P(\phi|_{\pl\bbar{X}})=0\}\]
is Fredholm if and only if $P\circ
P_{\bbar{\cH}_\pl}:\mc{H}_\pl\to {\rm ran}(P)$ is
Fredholm, and their indices agree. One of the main problems in this setting is to
construct Calder\'on projectors, there exist methods by
Wojciechowski \cite{BoW} which use the invertible double
construction, but a special product structure near the boundary has to
be assumed. Our purpose is to construct the Calder\'on projector
in a general setting for the Dirac operator using its conformal
covariance and the scattering theory of Dirac operators on
asymptotically hyperbolic manifolds developed in \cite{GMP}.

Let $x$ be the distance to the boundary, which is smooth near $\pl\bbar{X}$,
and modify it on a compact set of $X$ so that it becomes smooth on $\bbar{X}$, we still denote it by $x$.
Define a metric $g$ conformal to $\bbar{g}$ by
\[g:=x^{-2}\bbar{g},\]
this is a complete metric on the interior $X$ which is asymptotically hyperbolic.
The associated Dirac operator $D$ is related to $D_{\bar{g}}$ by the conformal law change
\[D_g=x^{\ndemi+1}D_{\bar{g}}x^{-\ndemi}.\]
Notice that this formula appears with a wrong exponent in several places
in the literature, e.g.\ \cite[Prop.\ 1.3]{Hitchin}, \cite[Thm.\ II.5.24]{LawMik}.
Let ${^0\Sigma}$ be the rescaled spin bundle defined in Section
\ref{AH}, then there is a canonical identification between
$\Sigma$ and ${^0\Sigma}$.  We
deduce that the Cauchy data space may also be given by
\[\mc{H}_\pl=\{(x^{-\ndemi}\sigma)|_{\pl\bbar{X}}; \sigma\in x^{\ndemi}C^{\infty}(\bbar{X},{^0\Sigma}), D_g\sigma=0\}.\]
Combining this and Theorem \ref{projector}, we obtain

\begin{theorem}\label{proj2}
The $L^2$-closure of the Cauchy data space $\bbar{{\mc{H}}}_\pl$ is
given by the range of $\mc{C}=\demi({\rm Id}+S(0))$ on
$L^2(\pl\bbar{X},{^0\Sigma})$, in particular,
$P_{\bbar{\mc{H}}_\pl}=\mc{C}$.
\end{theorem}
Remark that no assumption is needed on the geometry of
$(\bbar{X},\bbar{g})$ (this was needed for instance for the double
construction in \cite{BoW}). 

Another consequence of our
construction is that $S(0)$ anti-commutes with the endomorphism
${\rm cl}(\nu)$ of Section \ref{AH} and thus
\begin{prop}
The operator $\mc{C}$ satisfies $-\rm{cl}(\nu)\, \mc{C}\,
\rm{cl}(\nu)=\mathrm{Id}-\mc{C}$, in other words, the
$L^2$-closure of the Cauchy data space $\bbar{\mc{H}}_\pl$ is a
Lagrangian subspace in $L^2(\partial \bbar{X}, {^0\Sigma})$ with
respect to the symplectic structure $(v,w):=\langle {\rm cl}(\nu)
v, w \rangle_{h_0}$ for $v,w\in L^2(\partial \bbar{X}, {^0\Sigma})$
where $h_0=(\bbar{g})|_{T\pl \bbar{X}}$.
\end{prop}
\begin{proof}
The equality $-\rm{cl}(\nu)\, \mc{C}\,
\rm{cl}(\nu)=\mathrm{Id}-\mc{C}$ follows easily from $\rm{cl}(\nu)
S(0) = -S(0) \rm{cl}(\nu)$ since
\[
-\frac12 \rm{cl}(\nu) (\mathrm{Id}+ S(0))\rm{cl}(\nu) =\frac12
(\mathrm{Id}-\rm{cl}(\nu) S(0)\rm{cl}(\nu)) =\frac12 (\mathrm{Id}
-S(0)).
\]
This immediately implies that $\bbar{\mc{H}}_\pl$ and $\bbar{\mc{H}}^\perp$ are both 
isotropic subspaces in $L^2(\partial \bbar{X}, ^0\Sigma)$, which completes the proof.  
\end{proof}

\subsection{Calder\'on projector and the operator $K$}\label{caldproj}  By
Propositions \ref{p:harmonic} and \ref{span}, the extension map 
$K:C^\infty(\pl\bbar{X},{^0\Sigma})\to
C^{\infty}(\bbar{X},{^0\Sigma})$ from spinors on $M$ to harmonic spinors on $\bbar{X}$ is given by
\[K\psi=\demi x^{-\ndemi}E(0)\psi\]
where $E(0)$ is the operator defined in \eqref{defelasla} for the Dirac operator $D$ associated to $g=\bbar{g}/x^2$.
The adjoint $E(0)^*$ of $E(0)$ with respect to ${\rm dv}_g$
is a map from $\dot{C}^\infty(\bbar{X},{^0\Sigma})$ to
$C^\infty(\pl\bbar{X},{^0\Sigma})$ such that
\[ \int_X \cjg E(0)\varphi,\psi\cjd_g{\rm dv}_g=\int_{\pl\bbar{X}}\cjg \varphi,E(0)^*\psi\cjd_{h_0}{\rm
dv}_{h_0}\]  for all $\psi\in\dot{C}^\infty(\bbar{X},{^0\Sigma})$
and $\varphi\in C^{\infty}(\pl\bbar{X},{^0\Sigma})$. Here $h_0$
denotes the metric over $\pl\bbar{X}$ given by the restriction of
$\bbar{g}$ to the bundle $T\pl\bbar{X}$. Similarly the adjoint of
$K$ with respect to the metric ${\rm{dv}}_{\bbar{g}}$ satisfies
\[ \int_X \cjg K\varphi,\psi\cjd_g{\rm dv}_{{\bbar{g}}}=\int_{\pl\bbar{X}}\cjg \varphi,K^*\psi\cjd_{h_0}{\rm dv}_{h_0}\]
and since ${\rm dv}_{{g}}=x^{-(n +1)}{\rm dv}_{\bbar{g}}$, we
obtain
\[K^*=\demi E(0)^*x^{\ndemi+1}\]
where the adjoint for $E(0)$ is with respect to $g$ while the adjoint for $K$ is with respect to $\bar{g}$. 

The Schwartz kernels of $E(\la),E^*(\la)$ and $R_\pm(\la)$ are
studied in \cite{GMP}. They are shown to be polyhomogeneous
conormal on a blown-up space. Let us now describe them, by
referring the reader to the Appendix for what concerns blown-up
manifolds and polyhomogeneous conormal distributions. The first
space is the stretched product (see for instance \cite{MM,MaCPDE}
where it was first introduced)
\begin{align*}\bbar{X}\x_0\bbar{X}=[\bbar{X}\x\bbar{X};\Delta_{\pl}],&& 
\Delta_\pl:=\{(m,m)\in\pl\bbar{X}\x\pl\bbar{X}\}\end{align*}
obtained by blowing-up the diagonal $\Delta_\pl$ in the corner,
the blow-down map is denoted by $\beta:\bbar{X}\x_0\bbar{X}\to
\bbar{X}\x\bbar{X}$. This is a smooth manifold with corners which
has $3$ boundary hypersurfaces: the front face $\ff$ obtained from
blowing-up $\Delta_\pl$, and the right and left boundaries
$\rb$ and $\lb$  which respectively project down to
$\bbar{X}\x\pl\bbar{X}$ and $\pl\bbar{X}\x\bbar{X}$ under $\beta$.
One can similarly define the blow-ups
\begin{align}\label{stretched}
\bbar{X}\x_0\pl\bbar{X}:=[\bbar{X}\x\pl\bbar{X};\Delta_\pl] , &&
\pl\bbar{X}\x_0\bbar{X}:=[\pl\bbar{X}\x\bbar{X};\Delta_\pl]
\end{align}
which are manifolds with $1$ corner of codimension $2$ and $2$ boundary
hypersurfaces: the front face $\ff$ obtained from the blow-up and
the left boundary $\lb$ which projects to $\pl\bbar{X}\x\pl\bbar{X}$
for $\bbar{X}\x_0\pl\bbar{X}$, respectively the front face $\ff$ and right
boundary $\rb$ for $\pl\bbar{X}\x_0\bbar{X}$. We call
$\beta_l,\beta_r$ the blow-down maps of \eqref{stretched} and we
let
 $\rho_{\ff},\rho_{\lb}$ and $\rho_{\rb}$ be boundary defining functions of these
hypersurfaces in each case. Notice that the two spaces in
\eqref{stretched} are canonically diffeomorphic to the
submanifolds $\{\rho_{\rb}=0\}\subset\bbar{X}\x_0\bbar{X}$ and
$\{\rho_{\lb}=0\}\subset\bbar{X}\x_0\bbar{X}$. Like in Section 3.2
in \cite{GMP}, the bundle ${^0\Sigma}\boxtimes{^0\Sigma}^*$ lifts
smoothly to these 3 blown-up manifolds through $\beta, \beta_l$
and $\beta_r$, we will use the notation
\begin{align*}
\mc{E}:=\beta^*({^0\Sigma}\boxtimes{^0\Sigma}^*),&& \mc{E}_j:=\beta_j^*({^0\Sigma}\boxtimes{^0\Sigma}^*) \textrm{ for }j=l,r
\end{align*}
for these bundles. The interior diagonal in $X\x X$ lifts to a
submanifold $\Delta_{\iota}$ in $\bbar{X}\x_0\bbar{X}$ which
intersects the boundary only at the front face (and does so
transversally). Then it follows from \cite[Prop 3.2]{GMP}
that the resolvent $R_\pm(\la)$ has a Schwartz kernel
$R_\pm(\la;m,m')\in C^{-\infty}(\bbar{X}\x\bbar{X};\mc{E})$ which
lifts to $\bbar{X}\x_0\bbar{X}$ to a polyhomogeneous conormal
distribution on $\bbar{X}\x_0\bbar{X}\setminus \Delta_\iota$
\begin{equation}\label{e:R(0)}
\beta^*R_\pm(\la)\in
(\rho_{\rb}\rho_{\lb})^{\la+\ndemi}C^{\infty}(\bbar{X}\x_0\bbar{X}\setminus
\Delta_\iota;\mc{E}).\end{equation} 
Combined with Theorem \ref{proj2}, this structure result on $R_\pm(\la)$ implies 
\begin{cor}\label{vanishing residue}
The Schwartz kernel of the Calder\'on projector $P_{\bbar{\mc{H}}_\pl}$ associated to the Dirac operator 
has an asymptotic expansion in polar coordinates around the diagonal without log terms. In particular, the Wodzicki-Guillemin local residue density of $P_{\bbar{\mc{H}}_\pl}$  vanishes.
\end{cor}
\begin{proof} Using Theorem \ref{proj2}, it suffices to show that $S(0)$ has this property. 
From \cite[eq (4.10), Sec. 3]{GMP}, the kernel of $S(\la)$ is given outside the diagonal by 
\[S(\la;y,y')=i[(xx')^{-\la-\ndemi}R_+(\la;x,y,x',y')|_{x=x'=0}-(xx')^{-\la-\ndemi}R_-(\la;x,y,x',y')|_{x=x'=0}]\]
Since a boundary defining function $x'$ of
$\bbar{X}\x\pl\bbar{X}$ in $\bbar{X}\x\bbar{X}$ lifts to
$\beta^*x'=\rho_{\rb}\rho_{\ff}F$ for some $F>0$ smooth on
$\bbar{X}\x_0\bbar{X}$ (and similarly $\beta^*x=\rho_{\lb}\rho_{\ff}F$ for some smooth $F>0$), 
one can use \eqref{e:R(0)}  to obtain 
\[\beta^*((xx')^{-\la-\ndemi}R_\pm(\la)\in \rho_{\ff}^{-2\la-n}C^\infty(\bbar{X}\x_0\bbar{X}; \mc{E}).\]
Restricting to $x=x'=0$, $y\not=y'$ corresponds to restricting to the corner $\lb\cap \rb $ which is canonically diffeomorphic to
$M\x_0 M=[M\x M; \Delta_\pl]$ and thus 
the pull-back $\beta_\pl^* S(\la)$ of the kernel of $S(\la)$ has an expansion in polar coordinates at $\Delta_\pl$
with no log terms after setting $\la=0$.
\end{proof}

From \eqref{kernelE}, we deduce that the kernel $E(\la;m,y')$ of
$E(\la)$ lifts to
\[\beta_l^*E(\la)\in \rho_\lb^{\la+\ndemi}\rho_{\ff}^{-\la-\ndemi}C^{\infty}(\bbar{X}\x_0\pl\bbar{X};\mc{E}_l)\]
where we used the identification between $\{\rho_{\rb}=0\}\subset \bbar{X}\x_0\bbar{X}$ and $\bbar{X}\x_0\pl\bbar{X}$.
Here, obviously, this is the kernel of the operator acting from $L^2(M,{^0\Sigma};{\rm dv}_{h_0})$
to $L^2(X,{^0\Sigma};{\rm dv}_{g})$. We have
a similar description
\[\beta_r^*E^*(\lambda)\in\rho_\rb^{\la+\ndemi}\rho_{\ff}^{-\la-\ndemi}C^{\infty}(\pl\bbar{X}\x_0\bbar{X};\mc{E}_r).\]
So we deduce that the Schwartz kernel $K^*(y,x',y')\in
C^{\infty}(\pl\bbar{X}\x\bbar{X};{^0\Sigma}\boxtimes{^0\Sigma}^*)$
of $K^*$ with respect to the density $|{\rm dv}_{h_0}\otimes{\rm
dv}_{\bbar{g}}|=x^{n+1}|{\rm dv}_{h_0}\otimes{\rm dv}_{g}|$ lifts
through $\beta_r$ to
\begin{equation}\label{kernelK*}
\beta_{r}^*K^*=\demi \beta_r^*(x'^{-\ndemi}E^*(0))\in
\rho_{\ff}^{-n}C^{\infty}(\pl\bbar{X}\x_0\bbar{X};\mc{E}_r).
\end{equation}
Similarly, for $K$ we have
\begin{equation}\label{e:k1}
\beta_{l}^*K\in
\rho_{\ff}^{-n}C^{\infty}(\bbar{X}\x_0\pl\bbar{X};\mc{E}_l).
\end{equation}
When it is clear, we may omit $^0\Sigma$ in the notations
$L^2(\bbar{X},^0\Sigma, {\rm dv}_g)$, $L^2(\pl X, ^0\Sigma, {\rm
dv}_{h_0})$ for simplicity. Now we have
\begin{lemma}\label{Kbounded}  The operator $K$ is
bounded from $L^2(\pl\bbar{X},{\rm dv}_{h_0})$ to $L^2(\bbar{X},{\rm
dv}_{\bbar{g}})$, and so is its adjoint $K^*$ from
$L^2(\bbar{X},{\rm dv}_{\bbar{g}})$ to $L^2(\pl\bbar{X},{\rm
dv}_{h_0})$. The range of $K^*$ acting on $L^2(\bbar{X},{\rm
dv}_{\bbar{g}})$ is contained in $\bbar{\mc{H}}_\pl$ and the kernel
of $K$ contains $\bbar{\mc{H}}_\pl^\perp$.
\end{lemma}
\begin{proof}
It is shown in Lemma 4.7 of \cite{GMP} the following identity
\[R_+(0)-R_-(0)=-\frac{i}{2}(E_+(0)E_+(0)^*+ E_-(0)E_-(0)^*)=-\frac{i}{2}E(0)E(0)^*
\] as operators from $\dot{C}^\infty(\bbar{X},{^0\Sigma})$ to
$x^{\ndemi}C^{\infty}(\bbar{X},{^0\Sigma})$, so in particular this
implies that
\[KK^*=\demi ix^{-\ndemi}(R_+(0)-R_-(0))x^{\ndemi+1}\]
as operators. Using the isometry $\psi\to x^{-(n+1)/2}{\psi}$  from $L^2(X,{\rm
dv}_g)$ to $L^2(X,{\rm dv}_{\bbar{g}})$, we see that the operator
$KK^*$ is bounded on $L^2(\bbar{X},{\rm dvol}_{\bbar{g}})$ if and
only if $x^\demi(R_+(0)-R_-(0))x^{\demi}$ is bounded on
$L^2(X,{\rm dv}_g)$. Now by \eqref{e:R(0)}, the Schwartz kernel of
$x^{\demi}R_\pm(0)x'^\demi$ lifts on the blown-up space
$\bbar{X}\x_0\bbar{X}$ as a conormal function
\[\beta^*(x^{\frac12}R_\pm(0)x'^{\frac12}) \in
\rho_{\lb}^{\frac{n+1}{2}}\rho_{\rb}^{\frac{n+1}{2}}\rho_{\ff}\,C^{\infty}(\bbar{X}\x_0\bbar{X};\mc{E})\]
since $(xx')^\demi$ lifts to $\bbar{X}\x_0\bbar{X}$ to
$(\rho_{\rb}\rho_{\lb})^\demi\rho_\ff F$ for some $F>0$ smooth on
$\bbar{X}\x_0\bbar{X}$. We may then use Theorem 3.25 of Mazzeo
\cite{MaCPDE} to conclude that it is bounded on $L^2(X,{\rm
dv}_{g})$, and it is even compact according to Proposition 3.29 of
\cite{MaCPDE}. As a conclusion, $K^*$ is bounded from $L^2(X,{\rm
dv}_{\bbar{g}})$ to $L^2(\pl\bbar{X},{\rm dv}_{h_0})$ and $K$ is
bounded on the dual spaces. The fact that the range of $K^*$ is
contained in $\bbar{\mc{H}}_\pl$ comes directly from a density
argument and the fact that for all
$\psi\in\dot{C}^\infty(\bbar{X};{^0\Sigma})$,
$K^*\psi={-\demi i[x^{-\ndemi}(R_+(0)-R_-(0))(x^{\ndemi+1}\psi)]|_{\pl\bbar{X}}}$,
and $x^{-\ndemi}(R_+(0)-R_-(0))(x^{\ndemi+1}\psi)$ is a smooth harmonic spinor
of $D_{\bar{g}}$ on $\bbar{X}$.
\end{proof}

The operator $K^*K$ acts on $L^2(\pl\bbar{X},{\rm dv}_{h_0})$ as a compact operator, we actually obtain
\begin{lemma}\label{pseudo-1}
The operator $K^*K$ is a classical pseudo-differential operator of order $-1$ on $\pl\bbar{X}$ and its principal symbol is given by 
\[\sigma_{\rm pr}(K^*K)(y;\mu)=\frac{1}{4}
|\mu|^{-1}_{h_0}\Big({\rm Id}+i{\rm cl}(\nu){\rm cl}\Big(\frac{\mu}{|\mu|_{h_0}}\Big)\Big)\]
\end{lemma}
\begin{proof}
According to \eqref{e:k1} and Lemma \ref{rel0calculus1}, 
the operator $K=\demi x^{-\frac{n}{2}}E(0)$ is a log-free classical pseudodifferential operator 
in the class $I^{-1}_{\rm lf}(\bbar{X}\x M;\mc{E})$ in the terminology of Subsection \ref{interiortoboundary}, while $K^*$ is in the class 
$I^{-1}_{\rm lf}(M\x \bbar{X}\;\mc{E})$. 
We can therefore apply Proposition \ref{compositionKL} to deduce that $K^*K\in \Psi^{-1}(M;\mc{E})$ is a classical pseudo-differential operator of order $-1$ on $M$. Moreover, from Proposition \ref{compositionKL}, the principal symbol is given by 
\[\sigma_{K^*K}(y,\mu)=(2\pi)^{-2}\int_{0}^\infty \hat{\sigma}_{K^*}(y;-x,\mu).\hat{\sigma}_K(y;x,\mu)dx\]
where hat denotes Fourier transform in the variable $\xi$ and $\sigma_{K^*}(y,\xi,\mu),\sigma_{K}(y;\xi,\mu)$ are
the principal symbols of $K^*,K$.  We have to compute for $|\mu|$ large the integral above.
We know from \cite{GMP} that the leading asymptotic in polar coordinates around $\Delta_\pl$ (or equivalently the normal operator at the front face) of $K=\demi x^{-n/2}E(0)$ at the submanifold $\Delta_\pl$
is given in local coordinates by 
\[K(x,y,y+z)\sim \demi \pi^{-\frac{n+1}{2}}\Gamma\left(\frac{n+1}{2}\right)\rho^{-n-1}(x+\cl(\nu)\cl(z))\] 
where $\rho:=(x^2+|z|^2)^{1/2}$ is the defining function
for the front face of $\bbar{X}\x_0 M$. To obtain the symbol, we need to compute the 
inverse Fourier transform in $(x,z)$ variables of the homogeneous distribution $\rho^{-n-1}(x+\cl(\nu)\cl(z))$.
To do this, we use the analytic family of $L^1$ tempered
distributions $\omega(\lambda)=\rho^{-n-1+\lambda}$ for $\Re(\lambda)>0$. We have
\[\cF_{(x,z)\to (\xi,\mu)}(\omega(\lambda))= (2\pi)^{\frac{n+1}{2}} 2^{\lambda-\frac{n+1}{2}}
\frac{\Gamma\left(\frac{\lambda}{2}\right)}
{\Gamma\left(\frac{n+1-\lambda}{2}\right)} R^{-\lambda}\] for
$R:=|(\xi,\mu)|$. This allows us to compute
$\cF(x\omega(\lambda))$ and $\cF(z_j \omega(\lambda))$,
which turn out to be regular at $\lambda=0$. Thus by
setting $\lambda=0$ we get after a short computation
\[\sigma_K(y;\xi,\mu)=i(\xi^2+|\mu|^2)^{-1}(\xi+\cl(\nu)\cl(\mu)).\]
This gives $\sigma_{K^*}(y,\xi',\mu)=-i((\xi')^2+|\mu|^2)^{-1}(\xi'-\cl(\nu)\cl(\mu))$. 
Use the fact that the Fourier transform of the
Heaviside function is $\pi\delta-\frac{i}{\xi}$.  Then 
\[4\sigma_{K^*K}(y;\mu)=\pi^{-1} \int_{\rr} R^{-2}d\xi -\pi^{-2}i \int_{\rr^2}(RR')^{-2}(\xi\xi'+|\mu|^2
+(\xi'-\xi)\cl(\nu)\cl(\mu))\frac{d\xi d\xi'}{\xi-\xi'}\]
in the sense of principal value for $(\xi-\xi')^{-1}$.
The first term gives $|\mu|^{-1}$. In the second term, by symmetry in $\xi,\xi'$,
only the term $\pi^{-2}i\cl(\nu)\cl(\mu)\int_{\rr^2}(RR')^{-2}d\xi d\xi'$ contributes, and it
gives $i|\mu|^{-2}\cl(\nu)\cl(\mu)$. This ends the proof.
\end{proof}
In fact, we could also compute the principal symbol using the push-forward approach but the computation is 
slightly more technical.

We deduce easily from the two last lemmas
\begin{cor}\label{KstarKinv}
There exists a pseudo-differential operator of order $1$ on $\pl\bbar{X}$, denoted $(K^*K)^{-1}$ such that
$(K^*K)^{-1}K^*K=\mc{C}$.
\end{cor}
\begin{proof}
Using Lemmas  \ref{Kbounded}, \ref{pseudo-1},
we deduce that if $D_{h_0}$ is the Dirac operator on the boundary
$\pl\bbar{X}$ equipped with the metric
$h_0=\bbar{g}|_{T\pl\bbar{X}}$, then $A:= K^*K+\frac14({\rm Id}-\mc{C})({\rm Id}+D^2_{h_0})^{-\demi}({\rm Id}-\mc{C})$
is a classical pseudo-differential operator of order $-1$, and by Lemma \ref{pseudo-1} its principal
symbol on the cosphere bundle equals  ${\rm Id}$. Moreover, it is straightforward that $\ker A=0$  since $K$ is injective on
$\bbar{\mc H}_{\pl}$.
This implies that $A$ is elliptic and has a classical pseudo-differential inverse $B$ which is of order $1$.
Let us define $(K^*K)^{-1}:=B\mc{C}$, which is classical pseudo-differential of order 
$1$, then one has $(K^*K)^{-1}K^*K=(K^*K)^{-1}A=\mc{C}$.
\end{proof}

\subsection{The orthogonal projector on harmonic spinors on $\bbar{X}$}
We will construct and analyze the projector on the
$L^2(\bbar{X},{\rm dv}_{\bbar{g}})$-closure $\bbar{\mc{H}}(D_{\bar{g}})$
of
\[\mc{H}(D_{\bar{g}}):=\{\psi\in C^{\infty}(\bbar{X};{^0\Sigma}); D_{\bar{g}}\psi=0\}.\]
For this, let us now define the operator
\begin{equation}
P:=K(K^*K)^{-1}K^*
\end{equation}
which maps continuously $\dot{C}^\infty(\bbar{X};{^0\Sigma})$ to
$C^{\infty}(\bbar{X};{^0\Sigma})$. Since $K$ is bounded on
$L^2(\pl\bbar{X},{\rm dv}_{h_0})$, Lemma \ref{Kbounded} and
Corollary \ref{KstarKinv} imply easily the following
\begin{cor}\label{PK=K}
The operator $P$ satisfies $PK=K\mc{C}=K$ on $L^2(\pl\bbar{X}, {\rm
dv}_{h_0})$.
\end{cor}

We want to show that $P$ extends to a bounded operator on
$L^2(\bbar{X},{\rm dv}_{\bbar{g}})$ and study the structure of its
Schwartz kernel. We first use the following composition result
which is a consequence of Melrose's push-forward theorem \cite{Me}. 
The definition of polyhomogeneous functions and index sets is recalled
in Appendix \ref{appA}.

\begin{theorem}\label{structP}
The operator $P:=K(K^*K)^{-1}K^*$ has a Schwartz kernel in
$C^{-\infty}(\bbar{X}\x\bbar{X};({^0\Sigma}\boxtimes{^0\Sigma}^*)\otimes \Omega^\demi)$
on $\bbar{X}\x\bbar{X}$ which lifts to $\bbar{X}\x_0\bbar{X}$ through $\beta$ to  $k_P\beta^*(|{\rm dv}_{\bbar{g}}\otimes {\rm dv}_{\bbar{g}}|^\demi)$ with
\begin{align*}
k_P\in \mc{A}_{\rm phg}^{J_{\rm ff},J_{\rm rb},J_{\rm lb}}(\bbar{X}\x_0\bbar{X};\mc{E}), &&
J_{\rm ff}=-(n+1)\cup(-2,1)\cup(0,3),&& J_{\rm rb}=J_{\rm lb}=0
\end{align*}
where $|{\rm dv}_{\bbar{g}}\otimes{\rm dv}_{\bbar{g}}|$ is the
Riemannian density trivializing $\Omega(\bbar{X}\x\bbar{X})$ induced
by $\bbar{g}$.
\end{theorem}
\begin{proof}
We start by composing $A\circ B$ where $A:=(K^*K)^{-1}({\rm
Id}+D_{h_0}^2)^{-1}$ and $B:=({\rm Id}+D_{h_0}^2)K^*$. From
Corollary \ref{KstarKinv}, we know that $A$ is a classical
pseudo-differential operator on $M$ of order $-1$, so its kernel lifts to $M\x_0 M$
as a polyhomogeneous conormal kernel and its index set (as a
$b$-half-density) $E$ is of the form $-\ndemi+1+\nn_0\cup (\ndemi+\nn_0,1)$. 
Now, since the lift of vector fields on $M$
by the b-fibration $M\x_0\bbar{X}\to M\x\bbar{X}\to M$ is smooth,
tangent to the right boundary in $M\x_0\bbar{X}$ and transverse to
the front face $\ff$, we deduce that applying ${\rm Id}+D_{h_0}^2$
to $K^*$ reduces its order at $\ff$ by $2$ and leaves the index
set at $\rb$ invariant, so $({\rm Id}+D_{h_0}^2)K^*$ has a kernel
which lifts on $M\x_0\bbar{X}$ to an element in $\mc{A}_{\rm
phg}^{F_{\ff},F_{\rb}}(M\x_0\bbar{X};\mc{E}_r\otimes
\Omega_b^\demi)$ with
\begin{align*}F_{\ff}=-\ndemi-\frac{3}{2}, && F_{\rb}=\demi.\end{align*}
So using Lemma \ref{composition}, we deduce that $A\circ B$ has a kernel which lifts to $M\x_0\bbar{X}$
as an element in
\begin{align*}
\mc{A}_{\rm phg}^{H_{\ff},H_{\rb}}(M\x_0\bbar{X};\mc{E}_r\otimes \Omega_b^\demi),&&
H_{\ff}\subset (-\ndemi-\frac{1}{2}) \cup (\ndemi-\frac{3}{2},1)\cup (\ndemi+\demi,2),&&  H_{\rb}=-\frac{3}{2}\, \extunion \, \demi
\end{align*}
The index set $H_\rb$ must in fact be $\demi$ since the dual of this composition maps $C^\infty(M;{^0\Sigma})$ into
$C^\infty(\bbar{X};{^0\Sigma})$  (with respect to the density $|{\rm dv}_{\bar{g}}|^\demi$).
Now
the operator $K$ has a kernel lifted to $\bbar{X}\x_0M$ which is in
$\rho_{\ff}^{-\ndemi+\demi}\rho_{\lb}^{\demi}C^{\infty}(\bbar{X}\x_0M;\mc{E}_l\otimes \Omega_b^\demi)$
thus using Lemma \ref{composition} (and the same argument as above to show that the index set is $\demi$ at $\lb,\rb$),
we deduce that the lift $k_P$ of the Schwartz kernel
of $P$ is polyhomogeneous conormal on $\bbar{X}\x_0\bbar{X}$, and
the index set of $k_P$ satisfies  (as a b-half-density) 
\begin{align}\label{defJ}
 J_{\ff} =-\ndemi\cup (\ndemi-1,1)\cup (\ndemi+1,3),&&
J_{\lb}= J_{\rb}=\demi.
\end{align}
Now this completes
the proof since the lift of the half-density $|{\rm
dv}_{\bbar{g}}\otimes {\rm dv}_{\bbar{g}}|^\demi$ is of the form
$\rho_{\ff}^{\ndemi+1}\rho_{\lb}^\demi\rho_{\rb}^\demi
\mu_b^\demi$ where $\mu_b$ is a non vanishing smooth section of
$\Omega_b$.
\end{proof}

\begin{cor}
The operator $P=K(K^*K)^{-1}K^*$ is bounded on $L^2(\bbar{X},{\rm dv}_{\bbar{g}})$
and is the orthogonal projector on the $L^2$-closure of the set of
smooth harmonic spinors for $D_{\bar{g}}$ on $\bbar{X}$, that is,
$P=P_{\bbar{\mc{H}}}$.
\end{cor}
\begin{proof} Let $P':=x^{\frac{n+1}{2}}Px^{-\frac{n+1}{2}}$ acting on $\dot{C}^\infty(\bbar{X};{^0\Sigma})$,
then it suffices to prove that $P'$ extends to a bounded operator on $L^2(X,{\rm dv}_{g})$. But, in terms of half-densities,
the half density $|{\rm dv}_g\otimes {\rm dv}_g|^\demi$ is given by
$(xx')^{-\frac{n+1}{2}}|{\rm dv}_{\bbar{g}}\otimes {\rm dv}_{\bbar{g}}|^\demi$ and
Theorem  \ref{structP} shows that the Schwartz kernel of $P'$ lifts on $\bbar{X}\x_0\bbar{X}$
to a half-density $k_{P'}\beta^*(|{\rm dv}_g\otimes {\rm dv}_g|^\demi)$ where
\begin{align*}k_{P'}\in \mc{A}_{\rm phg}^{J'_{\rm ff},J'_{\rm rb},J'_{\rm lb}}(\bbar{X}\x_0\bbar{X};\mc{E})
&& J'_{\rm ff}\geq 0, && J'_{\rm rb}=J'_{\rm lb}=\frac{n+1}{2}.\end{align*}
It is proved in Proposition 3.20 of Mazzeo \cite{MaCPDE} that such
operators are bounded on $L^2(\bbar{X},{\rm dv}_g)$. To conclude,
we know from Corollary \ref{PK=K} that $P$ is the identity on the
range of $K$ acting on $C^{\infty}(\pl\bbar{X};{^0\Sigma})$, which
coincides with the space of smooth harmonic spinors for $D_{\bar{g}}$
on $\bbar{X}$, and we also know that $P$ vanishes on $\ker
(K^*)=\bbar{{\rm Im}(K)}^{\perp}$, so this achieves the proof.
\end{proof}

\section{Conformally covariant powers of Dirac operators and cobordism invariance of the index}

In this section, we define some conformally covariant differential
operators with leading part given by a power of the Dirac operator. 
The method is the same as in Graham-Zworski
\cite{GRZ}, using our construction of the scattering operator in
Section \ref{AH}. Since this is very similar to the case of
functions dealt with in \cite{GRZ}, we do not give much details. Let
$(X,g)$ be an asymptotically hyperbolic manifold with a metric $g$,
and let $x$ be a geodesic boundary defining function of $\pl\bbar{X}$
so that the metric has a product decomposition of the form
$g=(dx^2+h(x))/x^2$ near $\pl\bbar{X}$ as in \eqref{psig}. 
\begin{lemma}\label{finitemero}
Let $C(\la):=2^{-2\la}\Gamma(1/2-\la)/\Gamma(1/2+\la)$. If the
metric $g$ is even to infinite order, the operator
$\til{S}(\la):=S(\la)/C(\la)$ is finite meromorphic in $\cc$, and
it is holomorphic in $\{\Re(\la)\geq 0\}$. Moreover for
$k\in\nn_0$, the operator $L_k:=\til{S}(1/2+k)$ is a conformally
covariant self-adjoint differential operator on $\pl\bbar{X}$ with
leading part ${\rm cl}(\nu) D_{h_0}^{1+2k}$, and it depends only
on the tensors $\pl_x^{2j}h(0)$ in a natural way for $j\leq k$.
For $k=0$, one has $\til{S}(1/2)={\rm cl}(\nu)D_{h_0}$.
\end{lemma}
\begin{proof} The first statement is proved in Corollary 4.11 of \cite{GMP}.
The last statement about $\til{S}(1/2+k)$ is a consequence of the
construction of $\sigma_{\pm}(\la)$ in Proposition \ref{poisson},
by copying mutatis mutandis the proof of Theorem 1 of
Graham-Zworski \cite{GRZ}. Indeed, by construction, the term
$\sigma_{\infty,\pm}$ satisfying \eqref{sigmainfty} has a Taylor
expansion at $x=0$ of the form
\[\sigma_{\infty,\pm}(\la)=\psi+\sum_{j=1}^{k}x^{j}(p_{j,\la}\psi)+O(x^{k+1})\]
for all $k\in\nn$ where $p_{j,\la}$ are differential operators
acting on $C^{\infty}(\pl\bbar{X},{^0\Sigma})$ such that
$\frac{p_{j,\la}}{\Gamma(1/2-\la)}$ are holomorphic in
$\{\Re(\la)\geq 0\}$ and depend in a natural way only on the
tensors $(\pl_x^{\ell}h(0))_{\ell\leq j}$. Following Proposition
3.5 and Proposition 3.6 in \cite{GRZ}, the operator
$\textrm{Res}_{\la=1/2+k}S(\la)$ is also equal to $-{\rm
Res}_{\la=1/2+k}(p_{2k+1,\la})$. The computation of $\til{S}(1/2)$
is then rather straightforward by checking that
\[p_{1,\la}=-\frac{{\rm cl}(\nu)D_{h_0}}{2\la-1}\]
using the indicial equation \eqref{indicialeq} and the decomposition
\eqref{diracAH}.
\end{proof}

A first corollary of Lemma \ref{finitemero} is
the cobordism invariance of the index of the Dirac operator.
\begin{cor}
Let $D_{h_0}$ be the Dirac operator on a $2k$-dimensional closed
spin manifold $(M,h_0)$ which is the oriented boundary of a compact manifold
with boundary $(\bbar{X},\bbar{g})$. Let $D_{h_0}^+$ be the
restriction of $D_{h_0}$ to the sub-bundle of positive spinors
$\Sigma^+:=\ker(\omega-1)$, where $\omega$ 
is the Clifford multiplication by the volume element when $k$ is even, 
respectively $\omega=i\cl(\mathrm{vol}_{h_0})$ for $k$ odd. Then ${\rm
Ind}(D_{h_0}^+)=0$.
\end{cor}
\begin{proof}
By topological reasons, we may assume that $\bbar{X}$ is also spin and 
that the spin structure on $M$ is induced from that on $\bbar{X}$. 
Using the isomorphism between the usual spin bundle
$\Sigma(X)$ and the 0-spin bundle
${^0\Sigma}(X)$ in Section \ref{AH}, we see that $D_{h_0}$ can be
considered as acting in the restriction of the 0-spin bundle ${^0\Sigma}$ to $M$. Since
the odd-dimensional spin representation is chosen such that $\cl(\nu)=i\omega$, the $\pm i$
eigenspaces of ${\rm cl}(\nu)$ on ${^0\Sigma}(X)|_{M}$ correspond to the splitting
in positive, respectively negative spinors defined by
$\omega$ on $\Sigma(M)$. We have seen that $\til{S}(1/2)={\rm
cl}(\nu)D_{h_0}$. Then by the homotopy invariance of the index, 
it suffices to use the fact that $\til{S}(\la)$ is
invertible for all $\la$ except in a discrete set of $\cc$, which follows from
Lemma \ref{finitemero} and Proposition \ref{propofS}.
\end{proof}

We refer for instance to \cite{AS, Moro,Lesch,Nico,Brav}
for other proofs of the cobordism invariance of the index of $D^+$.

Now, let us consider $M$ a compact manifold equipped with a
conformal class $[h_0]$.  A $(n+1)$-dimensional
\emph{Poincar\'e-Einstein} manifold $(X,g)$ associated to
$(M,[h_0])$ is an asymptotically hyperbolic manifold with
conformal infinity $(M,[h_0])$ and such that the following extra condition
holds near the boundary $M=\pl\bbar{X}$
\begin{align*}{\rm Ric}(g)=-ng+O(x^{N-2}),&& N= \begin{cases}
\infty & \textrm{ if }n+1\textrm{ is even},\\
n & \textrm{ if }n+1{\textrm{ is odd.}} \end{cases}
\end{align*}
Notice that by considering the disjoint union $M_2:=M\sqcup M$ instead of $M$,
one sees  that either $M$ or $M_2$ can be realized as the boundary of a compact manifold with boundary
$\bbar{X}$.

Fefferman and Graham \cite{FGR,FGR2} proved that for any $(M,[h_0])$
which is the boundary of a compact manifold $\bbar{X}$, there exist
Poincar\'e-Einstein manifolds associated to $(M,[h_0])$. Moreover
writing $g=(dx^2+h(x))/x^2$ for a geodesic boundary defining function $x$,
the Taylor expansion of the metric $h(x)$ at $M=\{x=0\}$
is uniquely locally (and in a natural way) determined by $h_0=h(0)$ and the covariant derivatives
of the curvature tensor of $h_0$, but not on the Poincar\'e-Einstein metrics associated to $(M,[h_0])$.
If $M$ is spin, we can always construct a Poincar\'e-Einstein $(X:=[0,1]\x M,g)$ associated to $M_2$
with a spin structure induced naturally by that of $M$.

\begin{cor}\label{corcovdif}
If $(X,g)$ is a spin Poincar\'e-Einstein manifold associated to a
spin conformal manifold $(M,[h_0])$, then for $k\leq
N/2$ the operators $L_k=:-\cl(\nu)\til{S}(1/2+k)$ acting on
$C^{\infty}(M,{^0\Sigma})$ are self-adjoint natural (with respect
to $h_0$), conformally covariant differential operators of the
form $L_k=D_{h_0}^{2k+1}+\textrm{ lower order terms}$.
\end{cor}
Hence we can then always define the operators $L_k$ on $M_2=M\sqcup M$ 
and thus, since the construction is local and natural with respect to
$h_0$, this defines naturally $L_k$ on any $M$. As above, when
$(M,[h_0])$ is a boundary, the index of the restriction $L^\pm_k$
to ${^0\Sigma}_\pm=\ker(\omega\mp 1)$  (when $n$ is even) is always $0$. In general, the index
of $L_k^\pm$ is the index of $L^\pm_0$, which equals the $\hat{A}$-genus of $M$ by the
Atiyah-Singer index theorem \cite{AS}.

\subsection{The Dirac operator}
For $k=0$, the operator $L_0$, which is essentially the pole of
the scattering matrix at $\lambda=1/2$, is just the Dirac operator
$D_{h_0}$ on $(M,h_0)$ when the dimension of $M$ is even,
respectively two copies of $D_{h_0}$ when $\dim(M)$ is odd. 

\subsection{A conformally covariant operator of order $3$}
For $k=1$ in Corollary \ref{corcovdif} we get a conformally covariant
operator of order $3$ on any spin manifold of dimension $n\geq 3$,
with the same principal symbol as $D_{h_0}^3$.
\begin{theorem}
Let $(M,h_0)$ be a Riemannian spin manifold of dimension $n\geq 3$.
Then the differential operator of order $3$ acting on spinors
\[L_1:=D_{h_0}^3 -\frac{2\cl\circ\Ric_{h_0}\circ\nabla^{h_0}}{n-2}
+\frac{\scal_{h_0}}{(n-1)(n-2)}D_{h_0}-\frac{\cl(d(\scal_{h_0}))}{2(n-1)}\] 
is conformally covariant with respect to $h_0$ in the following sense:
if $\omega \in C^\infty(M)$  and $\hat{h}_0=e^{2\omega}h_0$, then 
\[\hat{L}_1=e^{-\frac{n+3}{2}\omega}L_1e^{\frac{n-3}{2}\omega}\]
where $\hat{L}_1$ is  defined as above but using the metric $\hat{h}_0$ instead of $h_0$.
\end{theorem}
\begin{proof}
The existence of the operator $L_1$ with the above covariance property
is already established, we are now going
to compute it explicitly. The asymptotic expansion 
of the Poincar\'e-Einstein metric $g=x^{-2}(dx^2+h_x)$ at the boundary is given in \cite{FGR2} by
\begin{align*}
\bbar{g}=x^2 g=dx^2+h_0-x^2P+O(x^4),&&
P=\tfrac{1}{n-2}\left(\Ric_{h_0}-\tfrac{\scal_{h_0}}{2(n-1)}\right).
\end{align*}
We trivialize the spinor bundle on $(\bbar{X},\bbar{g})$ from the boundary
using parallel transport along the gradient vector field $X:=\px$. Let us write the
limited Taylor
series of $D_{\bar{g}}$ in this trivialization:
\begin{equation}\label{bard}
D_{\bar{g}}=\cl(\nu)\px+D_0+xD_1+x^2D_2+O(x^3).
\end{equation}
Use the conformal change formula 
\begin{equation}\label{ccf}
D_g=x^{\frac{n+2}{2}}D_{\bar{g}}x^{-\frac{n}{2}}
\end{equation}
valid in dimension $n+1$. The idea from \cite{GRZ}
is to use the formal computation giving the residue of the scattering operator
at $\lambda=\frac32$ in terms of the $x^{\ndemi+3}\log(x)$ coefficient in the asymptotic expansion of formal solution to 
$(D_g-\frac{3}{2}i)\omega=0$ (the same method has been used in \cite{AuGu} for forms): there is a unique solution 
$\omega$ modulo $O(x^{\ndemi+3})$ of $(D_g-\frac{3}{2}i)\omega=O(x^{\ndemi+3})$ of the form
\begin{equation}\label{ansatz}
\omega=x^{\frac{n}{2}}\left(x^{-\frac{3}{2}}\omega^-_0+\sum_{j=1}^2
x^{j-\frac{3}{2}} \omega^\pm_j + x^{\frac{3}{2}} \log x \cdot \nu^+\right)+O(x^{\ndemi+3})
\end{equation}
and $\nu^+= C_k {\rm Res}_{\la=3/2}S(\la)\omega_0^-= C'_k \cl(\nu)L_1(\omega_0^-)$ for some non-zero constants $C_k,C'_k$. 
Since we know the principal term of $L_1$ is $D_{h_0}^3$, we can renormalize later and 
the constant $C'_k$ is irrelevant in the computation.
Recall that spinors in
the $\pm i$ eigenspaces of $\cl(\nu)$ are denoted with a $\pm$ symbol.

\begin{lemma}
The conformally covariant operator of order $3$ from Corollary
\ref{corcovdif} is given on by \begin{equation}\label{forml1}
L_1= D_0^3 +2\cl(\nu)(D_1D_0+D_0D_1)-4D_2.
\end{equation}
\end{lemma}
\begin{proof}
From \eqref{bard}, \eqref{ccf} and \eqref{ansatz} we derive by a
straightforward computation the identity \eqref{forml1} on negative spinors.
The same formula is obtained when we start with $\omega_0^+$,
so the lemma is proved.
\end{proof}

\begin{lemma}
The operators $D_1, D_2$ are given by
\begin{align*}
D_1=&\ -\frac{\scal_{h_0}\cl(\nu)}{4(n-1)},\\
-4D_2=&\ -2\cl\circ P\circ\nabla =-\tfrac{2}{n-2}\sum_{i,j=1}^n
\cl_i \Ric_{h_0;ij}\nabla_j +\frac{2\scal_{h_0} D_0}{2(n-1)(n-2)}.
\end{align*}
\end{lemma}
\begin{proof}
We write $\langle U,V\rangle$ for the scalar product with respect to the
$\bbar{g}$ metric, and $\nabla$ for the Riemannian connection. Notice that
for $U,V$ vectors tangent to the $\{x=x_0\}$ slices, and for $A$ defined by
the identity $P(U,V)=h_0(AU,V)$, we have
\begin{align*}
\langle U,V\rangle=h_0(U-x^2 AU,V)+O(x^4).
\end{align*}
Let $U,V$ be local vector fields on $M$. We first
extend them to be constant in the $x$ direction with respect to the
product structure $(0,\epsilon)_x\times M$. Then
\[\langle\nabla_X U,V\rangle=-xh_0(AU,V)+O(x^3)\]
which implies that the vector field
\[\tilde{U}:=U+\frac{x^2}{2} AU\]
is parallel with respect to $X$ modulo $O(x^3)$.
Let $(U_j)_{1\leq j\leq n}$ be a local orthonormal frame on $M$.
Then $(X,\tU_1,\ldots,\tU_n)$ is an orthonormal
frame on $(0,\eps)\x M$ up to order $O(x^4)$ and parallel with respect to $X$ to order $O(x^3)$. 
To compute the Dirac operator of $\bbar{g}$,
we use the trivialization of the spinor bundle ``from the boundary'' given by
the Gram-Schmidt orthonormalisation of this frame with respect to $\bar{g}$, which introduces an extra error term of order $O(x^4)$ 
(therefore harmless). Notice that
\[[\tU,\tilde{V}]=\widetilde{[U,V]}
-\frac{x^2}{2}\left(A[U,V]-[U,AV]-[AU,V]\right).\] 
Then we
compute from the Koszul formula
\begin{align*}
\nabla_X\tU_j=O(x^3),&&\nabla_X X=0,&& \nabla_{\tU_j}X=-xA\tU_j+O(x^3),
\end{align*}
\[\begin{split}
2\langle\nabla_{\tU_j}\tU_i,\tU_k\rangle=2h_0(\nabla^{h_0}_{U_j}U_i,U_k)
-\frac{x^2}{2}&\left\{h_0(A[U_j,U_i]-[U_j,AU_i]-[AU_j,U_i],U_k) \right.\\
&+h_0(A[U_k,U_j]-[U_k,AU_j]-[AU_k,U_j],U_i)\\
&\left. +h_0(A[U_k,U_i]-[U_k,AU_i]-[AU_k,U_i],U_j) \right\}+O(x^3).
\end{split}\]
We continue the computation at a point $p$ assuming that the frame $U_j$
is radially parallel from $p$, in particular at $p$ we have
$(\nabla^{h_0}_{U_j}U_i)(p)=0$, $[U_j,U_i](p)=0$ and $U_j(p)=\partial_j$
i.e., at $p$ the vector fields $U_j$ are just the coordinate vectors
of the geodesic normal coordinates. Then the coefficient of $\frac{x^2}{2}$ in
$2\langle\nabla_{\tU_j}\tU_i,\tU_k\rangle$ simplifies a lot, and we get at $p$
\[
2\langle\nabla_{\tU_j}\tU_i,\tU_k\rangle=2h_0(\nabla^{h_0}_{U_j}U_i,U_k)
-x^2(\partial_iA_{kj}-\partial_k A_{ij})+O(x^3).\]
From the local formula for the Dirac operator \cite[Eq 3.13]{BGV} we obtain
\begin{align*}D_{\bar{g}}=& \cl(X)\px+\cl_j(U_j+\frac{x^2}{2}AU_j) -\frac{1}{2}
\sum_{j,k=1}^n xA_{jk}\cl_j\cl(X)\cl_k
+\frac{1}{2} \sum_{i<k}h_0(\nabla^{h_0}_{U_j}U_i,U_k)\cl_j\cl_i\cl_k\\
&-\frac{x^2}{4} \sum_{j=1}^n\sum_{i<k}(\partial_iA_{kj}-\partial_k A_{ij})\cl_j\cl_i\cl_k+ O(x^3).
\end{align*}
It follows that $D_0$ is just the Dirac operator for $h_0$. For $D_1$, we could
additionally assume that at $p$, the vectors $U_j$ are eigenvectors of $A$,
thus $D_1=\frac{1}{2}\tr_{h_0}(A)\cl(X)$ which in view of the definition of $P$
(recall that $A$ is the transformation corresponding to $P$ with respect
to $h_0$) implies after a short computation the first formula of the lemma.

We also get
\[D_2=\frac{1}{2} \cl_j AU_j-\frac{1}{4}
\sum_{j=1}^n\sum_{i<k}(\partial_iA_{kj}-\partial_k A_{ij})\cl_j\cl_i\cl_k,\]
but in the first term the action of $U_j$ at $p$ clearly coincides with the
covariant derivative (the frame is parallel at $p$) so we get
the advertised formula.  As for the second term, it turns out to vanish
miraculously because of the coefficients inside $P$. Indeed, due to the
Clifford commutations we first check that the sum where $j,i,k$ are all
distinct vanishes. The remaining sum is given at $p$ by
\[\sum_{i,k} \cl_k(\partial_k A_{ii}-\partial_i A_{ik})\]
which in invariant terms reads
\[\cl(d(\tr_{h_0}(A)))+\cl(\delta^\nabla(A))\]
where $\delta^\nabla$ is the formal adjoint of the symmetrized covariant
derivative with respect to $h_0$. It is known that
\[\delta^\nabla\Ric_{h_0}+\frac{d(\scal_{h_0})}{2}=0,\]
and from
\begin{align*}
\tr_{h_0}(\Ric_{h_0})=\scal_{h_0}, && \tr_{h_0}(A)=\frac{\scal_{h_0}}{2(n-1)}, &&
\delta^\nabla(\scal_{h_0}\cdot I)=-d(\scal_{h_0})
\end{align*}
we get the result.
\end{proof}

This lemma ends the proof of the theorem by using \eqref{forml1}.
\end{proof}

\appendix
\section{ Polyhomogeneous conormal distributions, densities, blow-ups and index sets}\label{appA}
On a compact manifold with corners $\bbar{X}$, consider the set of
boundary hypersurfaces $(H_{j})_{j=1}^m$ which are codimension $1$
submanifolds with corners. Let $\rho_1, \dots, \rho_m$ be some
 boundary defining functions of these hypersurfaces. An
index set $\mc{E}=(\mc{E}_1,\dots, \mc{E}_m)$ is a subset of
$(\cc\x\nn_0)^m$ such that for each $M \in \rr$ the number of points
$(\beta, j) \in \mc{E}_{j}$ with $\Re(\beta) \leq M$ is finite, if
$(\beta, k) \in \mc{E}_{j}$ then $(\beta + 1, k)\in \mc{E}_{j} $,
and if $k > 0$ then also $(\beta, k-1) \in \mc{E}_{j}$. We define
the set
\[\dot{C}^\infty(\bbar{X}):=\{f\in C^\infty(\bbar{X}); f\textrm{ vanishes to all orders on each } H_{j}\}.\]
Its dual $C^{-\infty}(\bbar{X})$ is called the set of \emph{extendible distributions}
(the duality pairing is taken with respect to a fixed smooth $1$-density on $\bbar{X}$).
Conormal distributions on manifolds with corners were defined and
analyzed by Melrose \cite{Me,APS}, we refer the reader to these works for more details, but we give here
some definitions.
We say that an extendible distribution $f$ on a manifold with corners $X$ with boundary hypersurfaces $(H_1,\dots,H_m)$
is \emph{polyhomogeneous conormal}
(phg for short) at the boundary, with index set $\mc{E}=(\mc{E}_1,\dots,\mc{E}_m)$, if it is smooth in  the interior $X$, conormal
(i.e., if it remains in a fixed weighted $L^2$ space under repeated application of vector fields tangent
to the boundary of $\bbar{X}$) and if  for each $s \in \rr$ we have
\[
\left( \prod_{j=1}^m \prod_{\substack{(z, p) \in \mc{E}_j\\
 \text{ s.t. } \Re (z) \leq s}} (V_j - z) \right) f = O\big( (\prod_{j=1}^m \rho_j )^s \big)
\]
where  $V_j$ is a smooth vector field on $\bbar{X}$ that takes the form
$V_j = \rho_j \partial_{\rho_j} + O(\rho_j^2)$ near $H_j$.
This implies that $f$ has an asymptotic expansion in powers
and logarithms near each boundary hypersurface. In particular, near the interior of $H_j$,
we have
\[f = \sum_{\substack{{(z,p)} \in \mc{E}_j\\
\text{ s.t. } \Re (z) \leq s}} a_{(z,p)} \rho_j^z (\log \rho_j)^p
+O(\rho_j^s)\] for every $s \in \rr$, where $a_{(z,p)}$ is smooth
in the interior of $H_j$, and $a_{(z,p)}$ is itself
polyhomogeneous on $H_j$. The set of polyhomogeneous conormal
distributions with index set $\mc{E}$ on $\bbar{X}$ with values in
a smooth bundle $F\to \bbar{X}$ will be denoted by
\[\mc{A}^{\mc{E}}_{\rm phg}(\bbar{X};F).\]
Recall the operations of addition and extended union of two index
sets $E_1$ and $E_2$, denoted by $E_1 + E_2$ and $E_1 \extunion
E_2$  respectively:
\begin{equation}\begin{split}
&E_1 + E_2 = \{ (\beta_1 + \beta_2, j_1 + j_2) \mid (\beta_1, j_1) \in E_1 \text{ and } (\beta_2 , j_2) \in E_2 \} \\
&E_1\, \extunion\, E_2 = E_1 \cup E_2 \cup \{ (\beta, j) \mid \exists (\beta, j_1) \in E_1, (\beta, j_2) \in E_2 \text{ with } j = j_1 + j_2 + 1 \}.
\end{split}\end{equation}
In what follows, we shall write $q$ for the index set
$\{ (q + n, 0) \mid n = 0, 1, 2, \dots \}$
for any $q \in \rr$. For any index set $E$ and $q \in \rr$, we write $E \geq q$ if $\Re(\beta) \geq q$
for all $(\beta, j) \in E$ and if $(\beta, j) \in E$ and $\Re(\beta) = q$ implies $j = 0$.  
Finally we say that $E$ is integral if $(\beta, j) \in E$ implies that $\beta \in \mathbb{Z}$.

On $\bbar{X}$, the most natural densities are the $b$-densities
introduced by Melrose \cite{Me,APS}. The bundle
$\Omega_b(\bbar{X})$ of $b$-densities is defined to be
$\rho^{-1}\Omega(\bbar{X})$ where $\rho=\prod_{j}\rho_j$ is a total
boundary defining function and $\Omega(\bbar{X})$ is simply the
usual smooth bundle of densities on $\bbar{X}$. In particular a
smooth section of the $b$-densities bundle restricts canonically
on each $H_j$ to a smooth $b$-density on $H_j$. The bundle of
$b$-half-densities
is simply $\rho^{-\demi}\Omega^\demi(\bbar{X})$.

A natural class of submanifolds, called \emph{p-submanifolds}, of
manifolds with corners is defined in Definition 1.7.4 in
\cite{Melbook}. If $Y$ is a closed $p$-submanifold of $\bbar{X}$,
one can define the blow-up $[\bbar{X};Y]$ of $\bbar{X}$ around $Y$,
this is a smooth manifold with corners where $Y$ is replaced by
its inward pointing spherical normal bundle $S^+NY$ and a smooth
structure is attached using polar coordinates around $Y$. The new
boundary hypersurface is diffeomorphic to $S^+NY$ and is called
\emph{front face} of $[\bX;Y]$, there is a canonical smooth
blow-down map $\beta:[\bX;Y]\to \bX$ which is the identity outside
the front face and the projection $S^+NY\to Y$ on the front face.
See section 5.3 of \cite{Melbook} for details. The pull-back
$\beta^*$ maps continuously $\dot{C}^\infty(\bbar{X})$ to
$\dot{C}^{\infty}([\bbar{X};Y])$ and it is a one-to-one
correspondence, giving by duality the same statement for
extendible distributions.

\section{Compositions of kernels conormal to the boundary diagonal}\label{appB}

In this section, we introduce a symbolic way to describe conormal distributions associated
to the diagonal $\Delta_\pl$ inside the corner of $\bbar{X}\x\bbar{X}$, $\bbar{X}\x\pl\bbar{X}$,
or $\pl\bbar{X}\x\bbar{X}$. In particular, we compare the class of operators introduced by
Mazzeo-Melrose (the $0$-calculus) to a natural class of pseudo-differential operators
we define by using oscillatory integrals. We will prove composition results using both 
the push-forward Theorem of Melrose 
\cite{Me} and some classical symbolic calculus. 
We shall use the notations from the previous sections.

\subsection{Operators on $\bbar{X}$}

We say that an operator $K:\dot{C}^\infty(\bbar{X})\to C^{-\infty}(\bbar{X})$ is in the class $I^s(\bbar{X}\x\bbar{X},\Delta_\pl)$ if its Schwartz kernel $K(m,m')\in C^{-\infty}(\bbar{X}\x\bbar{X})$
is the sum of a smooth function $K_\infty\in C^\infty(\bbar{X}\x\bbar{X})$ and a singular kernel
$K_s$ supported near $\Delta_\pl$, which can be written in local coordinates $(x,y,x',y')$ near a point
 $(0,y_0,0,y_0)\in \Delta_\pl$ under the form (here $x$ is a boundary defining function on $\bbar{X}$
 and $y$ some local coordinates on $\pl\bbar{X}$ near $y_0$, and prime denotes the right variable version of them)
\begin{equation}\label{koscill}
K_s(x,y,x',y')=\frac{1}{(2\pi)^{n+2}}\int_\rr\int_\rr\int_{\rr^n} e^{-ix\xi-ix'\xi'-i(y-y')\mu}a(x,y,x',y';\xi,\xi',\mu)d\mu d\xi d\xi'
\end{equation}
where $a$ is a smooth classical symbol of order $s\in \rr$ in the sense that it satisfies for all multi-indices $\alpha,\alpha',\beta$
\[ |\pl_{m}^\alpha\pl_{m'}^{\alpha'}\pl^\beta_{\zeta}a(m,m';\zeta)|\leq C_{\alpha,\alpha',\beta}(1+|\zeta|^2)^{s-|\beta|}\]
where $m=(x,y)\in \rr^+\x \rr^n$ and $\zeta:=(\xi,\xi',\mu)\in \rr\x\rr\x\rr^n$. The integral in \eqref{koscill} makes sense
as an oscillatory integral: we integrate by parts a sufficient number $N$ of times in $\zeta$ to
get $\Delta_\zeta^N a(m,m';\zeta)$ uniformly $L^1$ in $\zeta$; of course we pick up a singularity of
the form $(x^2+{x'}^2+|y-y'|^2)^{-N}$ by this process but the outcome still makes sense
as an element in the dual of $\dot{C}^{\infty}(\bbar{X}\x\bbar{X})$.
If $\til{X}$ is an open manifold extending $\bbar{X}$, such a kernel can be extended to a kernel $\til{K}$ on the
manifold $\til{X}\x\til{X}$  so that $\til{K}$ is classically conormal to the embedded closed submanifold $\Delta_\pl$.
Therefore our kernels (which are extendible distributions on $\bbar{X}\x\bbar{X}$) can freely be considered as restriction of distributional kernels acting on a subset of functions of $\til{X}\x\til{X}$, i.e.
the set $\dot{C}^\infty(\bbar{X}\x\bbar{X})$
which corresponds to smooth functions with compact support included in $\bbar{X}\x\bbar{X}$.
Standard arguments of pseudodifferential operator theory show that we can require
that $K_s$ in charts is, up to a smooth kernel, of the form
\[K_s(x,y,x',y')=\frac{1}{(2\pi)^{n+2}}\int_\rr\int_\rr\int_{\rr^n} e^{-ix\xi-ix'\xi'-i(y-y')\mu}a(y;\xi,\xi',\mu)d\mu d\xi d\xi'.\] 
Indeed, it suffices to apply a Taylor expansion of $a(x,y,x',y';\zeta)$
at $\Delta_\pl=\{x=x'=y-y'=0\}$ and use integration by parts to show that
the difference obtained by quantizing these
symbols and the symbols of the form $a(y,\zeta)$ is given by smooth kernels.

We say that the symbol $a$ is \emph{classical of order $s$} if it has an asymptotic expansion as $\zeta:=(\xi,\xi',\mu)\to \infty $
\begin{equation}\label{expansion}
a(y;\zeta)\sim \sum_{j=0}^\infty a_{s-j}(y;\zeta)
\end{equation}
where $a_j$ are homogeneous functions of degree $s-j$ in $\zeta$.
It is clear from their definition that operators in $I^s(\bbar{X}\x\bbar{X},\Delta_\pl)$ have smooth kernels
on $(\bbar{X}\x\bbar{X})\setminus \Delta_\pl$. Let us consider the diagonal singularity of $K$ when its symbol
is classical.
\begin{lemma}\label{phgexp}
An operator $K_s\in I^{s-n-2}(\bbar{X}\x\bbar{X},\Delta_\pl)$ has a kernel which is the sum of a smooth kernel
together with a kernel which is smooth outside $\Delta_\pl$ and has an expansion at $\Delta_\pl$ in local coordinates $(x,y,x',y')$ of the form
\begin{equation}\label{Kxy}
K_s(x,y,x',y')\sim \begin{cases}
R^{-s}\sum_{j=0}^\infty R^jK^j(y,\omega) & \textrm{ if }s\notin \zz,\\
R^{-s}\sum_{j=0}^\infty R^jK^j(y,\omega)
+\log(R)\sum_{j=0}^\infty R^jK^{j,1}(y,\omega) & \textrm{ if }s\in \nn_0,\\
R^{-s}(\sum_{j=0}^\infty R^jK^j(y,\omega)
+\log(R)\sum_{j=0}^\infty R^jK^{j,1}(y,\omega)) & \textrm{ if }s\in -\nn,
\end{cases}
\end{equation}
where $R:=(x^2+{x'}^2+|y-y'|^2)^\demi$, $(x,x',y-y'):=R\omega$ and $K^j,K^{j,1}$ are smooth.
\end{lemma}
\begin{proof}
Assume $K$ has a classical symbol $a$ like in \eqref{expansion}.
First, we obviously have that for any $N\in\nn$, $K\in C^{N}(\bbar{X}\x\bbar{X})$ if $s<-N$.
Let us write $t=s-n-2$, then we remark that for all $y$, the homogeneous function
$a_{t-j}(y,.)$ has a unique homogeneous extension as a homogeneous
distribution on $\rr^{n+2}$ of order $t-j$ if $s\notin j-\nn_0$ (see \cite[Th 3.2.3]{Ho}),
and its Fourier transform is homogeneous of order $-s+j$. Clearly, $K(x,y,x',y')$ can be written as the Fourier transform
in the distribution sense in $\zeta$ of $A_N+B_N$ where for $N\in\nn$
\begin{align*} A_N(y,\zeta):=\sum_{j=0}^N a_{t-j}(y;\zeta), && B_N(y,\zeta):=a(y,\zeta)-A_N(y,\zeta).\end{align*}
Now $|\zeta|^{-s+N} B_N(y,\zeta)$ is in $ L^1(d\zeta)$ in  $|\zeta|>1$ thus
$\mc{F}_{\zeta\to Z}((1-\chi(\zeta))B_N(y,\zeta))$ is in $C^{[N-s]}$ with respect to all variables
if $\chi\in C_0^\infty(\rr^{n+2})$ equals $1$ near $0$, while the Fourier transform
$\mc{F}(\chi B_N)$ and $\mc{F}(\chi A_N)$ have the same regularity and are smooth since the convolution
of $\mc{F}(\chi)$ with a homogeneous function is smooth. This implies the expansion of $K$ at the diagonal when
$t\notin \zz$.

For the case $t\in \zz$, this is  similar but a bit more complicated.
We shall be brief and refer to Beals-Greiner
\cite[Chap 3.15]{BeGr} for more details (this is done for the Heisenberg calculus there but their proof
obviously contains the classical case). Let us denote $\delta_\la$ the action of dilation by $\la\in\rr^+$ on the space $\mc{S}'$ of tempered distributions on $\rr^{n+2}$, then any homogeneous function $f_k$ of degree $-n-2-k\in -n-2-\nn_0$
on $\rr^{n+2}$ can be extended  to a distribution $\til{f}_k\in\mc{S}'$ satisfying
\begin{equation}\label{deltala}
\delta_\la(\til{f}_k)= \la^{-n-2-k} \til{f}_k + \la^{-n-2-k}\log (\la) P_k
\end{equation}
for some $P_k\in\mc{S}'$ of order $k$ supported at $0$. This element $P_k$ is zero if and only if
$f_k$ can be extended as a homogeneous distribution on $\rr^{n+2}$, or equivalently
\begin{align}\label{homogene}
\int_{S^{n+1}}f_k(\omega)\omega^\alpha d\omega=0,&& \forall \alpha\in \nn_0^{n+2}\textrm{ with }  |\alpha|=k.
\end{align}
According to Proposition 15.30  of \cite{BeGr}, the distribution $\til{f}_k$ has its Fourier transform which can be written outside $0$ as
\[\mc{F}(\til{f}_k)(Z)=L_k(Z)+M_k(Z)\log |Z| \]
where $L_k$ is a homogeneous function of degree $k$ on $\rr^{n+2}\setminus\{0\}$
and $M_k$ a homogeneous polynomial of degree $k$. Thus reasoning as above when $t\notin \zz$, this concludes the proof.
It can be noted from \eqref{homogene}�
that in the expansion at $\Delta_\pl$ in \eqref{Kxy}, one has $K^{j,1}=0$ for all $j=0,\dots, k$ for some $k\in\nn$
if the symbols satisfy the condition
\begin{align}\label{condition}
\int_{S^{n+1}}a_{-n-2-j}(y,\omega)\omega^\alpha d\omega=0,&& \forall \alpha\in \nn_0^{n+2}\textrm{ with }  |\alpha|=j
\end{align}
for all $ j=0,\dots,k$ and all $y\in M$. Using the expression of the symbol expansion after a change of coordinates, it is
straightforward to check that this condition is invariant with respect to the choice of coordinates.
\end{proof}
A consequence of this Lemma (or another way to state it) is that if $K\in I^{s-n-2}(\bbar{X}\x\bbar{X},\Delta_\pl)$
is classical, then its kernel
 lifts to a conormal polyhomogeneous distribution on the manifold with corners
$\bbar{X}\x_0\bbar{X}$ obtained by blowing-up $\Delta_\pl$ inside $\bbar{X}\x\bbar{X}$ and
\begin{equation}\label{betaK}
\beta^*K\in C^\infty(\bbar{X}\x_0\bbar{X})+
\begin{cases}
\rho_{\ff}^{-s}C^\infty(\bbar{X}\x_0\bbar{X}) & \textrm{ if }s\notin \zz,\\
\rho_{\ff}^{-s}C^\infty(\bbar{X}\x_0\bbar{X})+\log(\rho_{\ff})C^\infty(\bbar{X}\x_0\bbar{X}) & \textrm{ if }s\in \nn_0,\\
\rho_{\ff}^{-s}(C^\infty(\bbar{X}\x_0\bbar{X})+\log(\rho_{\ff})C^\infty(\bbar{X}\x_0\bbar{X})) & \textrm{ if }s\in -\nn.
\end{cases}
\end{equation}
Therefore $I^s(\bbar{X}\x\bbar{X},\Delta_\pl)$ is a subclass of the full $0$-calculus of Mazzeo-Melrose \cite{MM}, in particular with no interior diagonal singularity. Let us make this more precise:
\begin{lemma}\label{rel0calculus}
Let $\ell\in-\nn$, then a classical operator $K\in I^{\ell}(\bbar{X}\x\bbar{X},\Delta_\pl)$
with a local symbol expansion \eqref{expansion}
has a kernel which lifts to $\beta^*K\in \rho_{\rm ff}^{-\ell-n-2}C^\infty(\bbar{X}\x_0\bbar{X})+C^\infty(\bbar{X}\x_0\bbar{X})$
if the symbol satisfies the
condition \eqref{condition} for all $j\in\nn_0$. Conversely, if $K\in C^{-\infty}(\bbar{X}\x\bbar{X})$
is a distribution which lifts to $\beta^*K$ in $\rho_{\rm ff}^{-\ell-n-2}C^\infty(\bbar{X}\x_0\bbar{X})+C^\infty(\bbar{X}\x_0\bbar{X})$,
then it is  the kernel of a classical operator in $I^{-n-2}(\bbar{X}\x\bbar{X},\Delta_\pl)$ with a symbol satisfying \eqref{condition} for
all $j\in\nn_0$.
\end{lemma}
\begin{proof}
Let us start with the converse: we can extend smoothly the kernel $\beta^*K$ to the blown-up space
$[\til{X}\x\til{X}, \Delta_\pl]$ where $\til{X}$ is an open manifold extending smoothly $\bbar{X}$. Then the extended
function has an expansion to all order in polar coordinates
$(R,\omega)$ at $\{R=0\}$ (i.e., around $\Delta_\pl$) where $R=(x^2+{x'}^2+|y-y'|^2)^\demi$ and $R\omega=(x,x',y-y')$
\begin{align*}K(x,y,x',y')- \sum_{j=0}^k R^{-\ell -n-2+j}K^j(y,\omega) \in C^{k}(\bbar{X}\x\bbar{X}),&& \forall k\in\nn\end{align*}
for some smooth $K^j$, in particular using Fourier transform in $Z=(x,x',y-y')$ one finds that for all $k\in\nn$, there exists a classical symbol $a^{k}(y,\zeta)$
\[K(x,y,x',y')-\frac{1}{(2\pi)^{n+2}}\int e^{ix\xi+ix'\xi'+i(y-y')\mu} a^k(y;\xi,\xi',\mu)d\xi d\xi' d\mu \in C^k(\bbar{X}\x\bbar{X})\]
with $a^k$ being equal to  $\sum_{j=0}^ka^k_j(y;\zeta)$ when $|\zeta|>1$
for some homogeneous functions $a^k_j$ of degree $\ell-j$. Moreover, the $a^k_j$ can be extended
as homogeneous distribution on $\rr^{n+2}$ since they are given by Fourier transforms of
the homogeneous distributions $K^j(y,Z)$ in the variable $Z$. Using
that a homogeneous function on $\rr^{n+2}\setminus \{0\}$ which extends as a homogeneous distribution on $\rr^{n+2}$
has no $\log \la$ terms in \eqref{deltala}, or equivalently satisfies \eqref{homogene}, this ends one way.

To prove the first statement, it suffices to consider the kernel in local coordinates and locally $\beta^*K$
has the structure \eqref{betaK} with no $\log(\rho_\ff)$ if the local symbol satisfies \eqref{condition}.
Notice that having locally the structure $\rho_\ff^{-s}C^\infty(\bbar{X}\x\bbar{X})$ for a function is a property which is
independent of the choice of coordinates. But from what we just proved above, this implies that in any choice of coordinates
the local symbol satisfies \eqref{condition}.
\end{proof}

We shall call the subclass of operators in Lemma \ref{rel0calculus} the class of \emph{log-free classical operators}
of order $\ell\in-\nn$, and denote it $I_{\rm lf}^\ell(\bbar{X}\x\bbar{X},\Delta_\pl)$.

For (log-free if $s\in-\nn$) classical operators in $I^s(\bbar{X}\x\bbar{X},\Delta_\pl)$,
there is also a notion of \emph{principal symbol} which is  defined as a homogeneous section
of degree $s$ of the conormal bundle $N^*\Delta_\pl$: if $a$ has an expansion
$a(y,\zeta)\sim \sum_{j=0}^\infty a_{s-j}(y,\zeta)$ as $\zeta\to \infty$ with $a_{s-j}$ homogeneous of degree
$s-j$ in $\zeta$, then the principal symbol is given by $\sigma_{\rm pr}(K)=a_s$.
The principal symbol is actually not invariantly defined
if one considers $K$ as an extendible distribution on $\bbar{X}\x\bbar{X}$ :
if $a(y,\zeta)$ and $a'(y,\zeta)$ are two classical symbols
for the kernel $K$, then if $Z=(x,x',z)$
\[\mc{F}_{\zeta \to Z} (a_s(y,\zeta)-a_s'(y,\zeta))=0 \textrm{ when }x>0 \textrm{ and } x'>0,\]
thus it is defined only up to this equivalence relation.

To make the correspondence with the 0-calculus of Mazzeo-Melrose \cite{MM}, we recall that the normal
operator of an operator $K\in C^\infty(\bbar{X}\x_0\bbar{X})$ is given by the restriction to the front face: if $y\in \Delta_\pl$,
$N_y(K):= K|_{\ff_y}$ where $\ff_y$ is the fiber at $y$ of the unit interior pointing spherical normal bundle
$S^+N\Delta_\pl$ of $\Delta_\pl$ inside $\bbar{X}\x\bbar{X}$, then we remark that the normal operator at $y\in\Delta_\pl$ of
an admissible operator $K\in I_{\rm lf}^{-n-2}(\bbar{X}\x\bbar{X};\Delta_\pl)$
is given by the homogeneous function of degree $0$
on $\rr^+\x\rr^+\x \rr^n\simeq \ff_y\x\rr_+$
\[N_y(K)(Z)=\mc{F}_{\zeta\to Z}(\sigma_{\rm pr}(K)(y,\zeta)).\]

\subsection{Operators from $\bbar{X}$ to $\pl\bbar{X}$ and conversely}\label{interiortoboundary}
We define operators in $I^s(\bbar{X}\x\pl\bbar{X},\Delta_\pl)$ and $I^s(\pl\bbar{X}\x \bbar{X},\Delta_\pl)$ by saying that their
respective distributional kernels are the sum of a smooth kernel on $\bbar{X}\x\pl\bbar{X}$ (resp.
$\pl\bbar{X}\x \bbar{X}$) and of a singular kernel $K_s\in C^{-\infty}(\bbar{X}\x\pl\bbar{X})$
(resp. $L_s\in C^{-\infty}(\pl\bbar{X}\x\bbar{X})$) supported near  $\Delta_\pl$ of the form (in local coordinates)
\begin{equation}\label{operatorsKL}
\begin{split}
&K_s(x,y,y')=\frac{1}{(2\pi)^{n+1}}\int e^{-ix\xi+i(y-y')\mu}a(y';\xi, \mu)d\xi d\mu ,\\ 
& L_s(y,x',y')=\frac{1}{(2\pi)^{n+1}}\int e^{ix'\xi'+i(y-y')\mu}b(y;\xi', \mu)d\xi' d\mu
\end{split}\end{equation}
with $a$ and $b$ some smooth symbols
\begin{align*} |\pl_y^\alpha \pl_{\zeta}^\beta a(y,\zeta)|\leq C_{\alpha,\beta}\cjg \zeta\cjd^{s-|\beta|},
&& |\pl_y^\alpha \pl_{\zeta}^\beta b(y,\zeta)|\leq C_{\alpha,\beta}\cjg \zeta\cjd^{s-|\beta|}\end{align*}
for all $\alpha,\beta$. We shall say they are \emph{classical} if their symbols have an expansion in homogeneous
functions at $\zeta\to \infty$, just like above for operators on $\bbar{X}$.
It is easy to see that such operators map respectively $\dot{C}^\infty(\bbar{X})$ to $C^\infty(\pl\bbar{X})$
and $C^\infty(\pl\bbar{X})$ to $C^{-\infty}(\bbar{X})\cap C^\infty(X)$.

Using the exact same arguments as for operators on $\bbar{X}$, we have the following
\begin{lemma}\label{rel0calculus1}
Let $\ell\in-\nn$, then a classical operator $K\in I^{\ell}(\bbar{X}\x\pl\bbar{X},\Delta_\pl)$ with a local symbol expansion
$a(y,\zeta)\sim \sum_{j=0}^\infty a_{-n-1-j}(y,\zeta)$
has a kernel which lifts to $\beta_1^*K\in \rho_{\rm ff}^{-\ell-n-1}C^\infty(\bbar{X}\x_0\pl\bbar{X})+C^\infty(\bbar{X}\x_0\pl\bbar{X})$ if
\begin{align}\label{condition2}
\int_{S^{n}}a_{-n-1-j}(y,\omega)\omega^\alpha d\omega=0,&& \forall \alpha\in \nn_0^{n+1}\textrm{ with }  |\alpha|=j
\end{align}
for all $j\in\nn_0$. Conversely, if $K\in C^{-\infty}(\bbar{X}\x\bbar\pl{X})$
is a distribution which lifts to $\beta_1^*K$ in $\rho_{\rm ff}^{-\ell-n-1}C^\infty(\bbar{X}\x_0\pl\bbar{X})+C^\infty(\bbar{X}\x_0\pl\bbar{X})$, then it is
the kernel of a classical operator in $I^{\ell}(\bbar{X}\x\pl\bbar{X},\Delta_\pl)$ with a symbol satisfying \eqref{condition} for
all $j\in\nn_0$. The symmetric statement holds for operators in $I^{\ell}(\pl\bbar{X}\x\bbar{X},\Delta_\pl)$.
\end{lemma}

We shall also call the operators of Lemma \ref{rel0calculus1} \emph{log-free classical operators} and denote this class by
$I^\ell_{\rm lf}(\bbar{X}\x\pl\bbar{X},\Delta_\pl)$ and $I^\ell_{\rm lf}(\pl\bbar{X}\x\bbar{X},\Delta_\pl)$.

Notice that, since the restriction of a function in $C^\infty(\bbar{X}\x_0\bbar{X})$ to the right boundary gives
a function in $C^\infty(\bbar{X}\x_0\pl\bbar{X})$, we deduce that an operator
$I^{-n-2}(\bbar{X}\x\bbar{X},\Delta_\pl)$ satisfying condition \eqref{condition}
induces naturally (by restriction to the boundary on the right variable) 
an operator  in $I^{-n-1}(\bbar{X}\x\pl\bbar{X},\Delta_\pl)$ satisfying \eqref{condition2}. This can also be seen by considering the
oscillatory integrals restricted to $x'=0$ but it is more complicated to prove.

\subsection{Compositions}

We start with a result on the composition of operators mapping from $\bar{X}$ to $M$ with operators 
mapping $M$ to $M$ or $M$ to $\bar{X}$. This is will be done using the push-forward theorem of Melrose  \cite[Th. 5]{Me}
\begin{prop}\label{composition}
Let $A:C^\infty(M;{^0\Sigma}\otimes \Omega^\demi)\to
C^\infty(M;{^0\Sigma}\otimes \Omega^\demi)$ be a
pseudo-differential operator of negative order with lifted kernel
 in $\mc{A}_{\rm phg}^{{E_{\rm ff}}}(M\x_0 M;
\mc{E}\otimes \Omega_b^\demi)$. Let $B:
\dot{C}^{\infty}(\bbar{X};{^0\Sigma}\otimes\Omega_b^\demi) \to
C^\infty(M;{^0\Sigma}\otimes\Omega^\demi)$ be an operator with
lifted kernel in $\mc{A}_{\rm phg}^{F_{\rm ff},F_{\rm
rb}}(M\x_0\bbar{X};\mc{E}_r\otimes \Omega_b^\demi)$ and let
$C:C^{\infty}(M;{^0\Sigma}\otimes\Omega^\demi)\to
C^{-\infty}(\bbar{X};{^0\Sigma}\otimes \Omega_b^\demi)$ be an
operator with lifted kernel on $\mc{A}_{\rm phg}^{G_{\rm
ff},G_{\rm lb}}(\bbar{X}\x_0{M}; \mc{E}_l\otimes \Omega_b^\demi)$.
Then the Schwartz kernels of $A\circ B$ and $C\circ B$ lift to
polyhomogeneous conormal kernels
\begin{align*}
k_{A\circ B}\in \mc{A}_{\rm phg}^{H_{\rm ff},H_{\rm rb}}(M\x_0\bbar{X};
\mc{E}_r\otimes \Omega_b^\demi),&&
k_{C\circ B}\in \mc{A}_{\rm phg}^{I_{\rm ff},I_{\rm lb},I_{\rm rb}}(\bbar{X}\x_0\bbar{X};
\mc{E}\otimes \Omega_b^\demi)\end{align*}
and the index sets satisfy
\begin{align*}
&H_{\rm ff}=(E_{\rm ff}+ F_{\rm ff}+\ndemi)\, \extunion \,(F_{\rm rb}+\ndemi),& H_{\rm rb}=F_{\rm rb}\,\extunion \,(F_{\rm ff}+\ndemi)&& \\
&I_{\rm ff}= (F_{\rm ff}+G_{\rm ff}+\ndemi)\,\extunion\, (F_{\rm rb}+G_{\rm lb}+\ndemi),&
I_{\rm lb}= G_{\rm lb}\,\extunion\,(G_{\rm ff}+\ndemi), && 
I_{\rm rb}=F_{\rm rb}\,\extunion \,(F_{\rm ff}+\ndemi).
\end{align*}
\end{prop}
\begin{proof} The proof is an application of Melrose push-forward theorem. Let us discuss
first the composition $A\circ B$. We denote by $\Delta$ both the
diagonal in $M\x M$ and the submanifold $\{(m,m')\in
M\x\bbar{X};m=m'\}$, by $(\pi_j)_{j=l,c,r}$ the canonical
projections of $M\x M\x\bbar{X}$ obtained by projecting-off the $j$
factor (here $l,c,r$ mean left, center, right), and let
\begin{align*}\Delta_3:=\{(m,m',m'')\in M\x M\x\bbar{X}; m=m'=m''\},
&& \Delta_{2,j}=\pi_j^{-1}(\Delta) \textrm{ for }j=l,c,r.\end{align*}
The triple space $M\x_0M\x_0\bbar{X}$ is the iterated blow-up
\begin{equation}\label{triplespace}
M\x_0M\x_0\bbar{X}:=[M\x M\x\bbar{X};\Delta_3,\Delta_{2,l},\Delta_{2,c},\Delta_{2,r}].
\end{equation}
The submanifolds to blow-up are $p$-submanifolds, moreover $\Delta_3$ is contained in each $\Delta_{2,j}$
and the lifts of $\Delta_{2,j}$ to the blow-up $[M\x M\x \bbar{X};\Delta_3]$ are disjoint. Consequently
(see for instance \cite[Lemma 6.2]{GuHa}) the order of blow-ups can be commuted and the canonical projections $\pi_j$ lift to maps
\[
\beta_l:M\x_0 M\x_0\bbar{X}\to M\x_0\bbar{X},\, \beta_c:M\x_0 M\x_0\bbar{X}\to M\x_0\bbar{X}, \,
\beta_r:M\x_0 M\x_0\bbar{X}\to M\x_0 M\]
which are $b$-fibrations.
The manifold $M\x_0M\x_0\bbar{X}$ has $5$ boundary hypersurfaces,
the front face $\ff'$ obtained by blowing up $\Delta_3$, the faces
${\rm lf},{\rm cf}, {\rm rf}$ obtained from the respective blow-up
of $\Delta_{2,l},\Delta_{2,c},\Delta_{2,r}$ and finally the face
$\rb'$ obtained from the lift of the original face $M\x M\x
M\subset M\x M\x \bbar{X}$. We denote by $\rho_{f}$ a smooth
boundary defining function of the face $f\in\{\ff',{\rm rf},{\rm
cf},{\rm lf},\rb'\}$. If $k_A$ and $k_B$ are the lifted kernel of
$A$ and $B$ to respectively $M\x_0M$ and $M\x_0\bbar{X}$ then it is
possible to write the composition as a push-forward
\[k_{A\circ B}.\mu ={\beta_c}_*\Big(\beta_r^*k_A.\beta_l^*k_B.\beta_c^*\mu\Big)\] if $\mu\in
C^{\infty}(M\x_0\bbar{X};\mc{E}_r\otimes \Omega_b^\demi)$. An easy
computation shows that a smooth b-density $\omega$ on $M\x
M\x\bbar{X}$ lifts through $\beta$ to an element
\[\beta^*\omega\in \rho_{\ff'}^{2n}(\rho_{\rm lf}\rho_{\rm rf}\rho_{\rm cf})^n C^\infty(M\x_0M\x_0\bbar{X};
\Omega_b)\] so by considering the lifts through
$\beta_{l},\beta_c,\beta_r$ of boundary defining functions in
 $M\x_0\bbar{X}$, $M\x_0\bbar{X}$ and
$\bbar{X}\x_0\bbar{X}$  respectively we deduce that there is some
index set $K=(K_{\ff'},K_{\rb'},K_{\rm lf},K_{\rm rf},K_{\rm cf})$
such that
\begin{align*}
\lefteqn{\beta_r^*k_A.\beta_l^*k_B.\beta_c^*\mu\in
\mc{A}_{\rm phg}^{K}(M\x_0M\x_0\bbar{X};\Omega_b),}\\
&K_{\ff'}=E_{\ff}+ F_{\ff}+\ndemi,&& K_{\rb'}=F_{\rb},&&  K_{\rm lf}=F_{\ff}+\ndemi,&&
K_{\rm rf}=E_{\ff}+\ndemi, && K_{\rm cf}=F_{\rb}+\ndemi.
\end{align*}
Then from the push-forward theorem of Melrose \cite[Th. 5]{Me}, we obtain that
\begin{align*}
\lefteqn{(\beta_c)_*(\beta_r^*k_A.\beta_l^*k_B.\beta_c^*\mu)\in \mc{A}_{\rm phg}^{H_{\ff},H_{\rb}}(M\x_0\bbar{X},\Omega_b),}\\
&H_{\ff}=(E_{\ff}+ F_{\ff}+\ndemi)\, \extunion \,(F_{\rb}+\ndemi),
&& H_{\rb}=F_{\rb}\,\extunion \,(F_{\ff}+\ndemi)&&
\end{align*}
and this shows the first composition result for $A\circ B$. Remark that to apply \cite[Th.5]{Me}, we need the index of
$K_{\rm rf}>0$, i.e., $E_{\ff}+n/2>0$, but this is automatically satisfied with our assumption that
$A$ is a pseudodifferential operator of negative order on $M$.

The second composition result is very similar, except that there are more boundary faces to consider.
One defines $\Delta_3:=\{(m,m',m'')\in \bbar{X}\x M\x\bbar{X}; m=m'=m''\}$ and let
\[\Delta_{2,j}=\{(m_l,m_c,m_r)\in\bbar{X}\x M\x\bbar{X}; m_i=m_k \textrm{ if }j\notin\{i,k\}\}\]
similarly as before. The triple space is defined like
\eqref{triplespace}, it has now $6$ boundary faces which we denote
as in the case above but with the additional face, denoted $\lb'$,
obtained from the lift of the original boundary $M\x M\x\bbar{X}$.
The same arguments as above show that the
canonical projections from $\bbar{X}\x_0 M\x_0\bbar{X}$ obtained by
projecting-off one factor lift to b-fibrations
$\beta_r,\beta_l,\beta_c$ from the triple space to $\bbar{X}\x_0M$,
$M\x_0\bbar{X}$ and $\bbar{X}\x_0\bbar{X}$. Like for the case above,
one has to push-forward a distribution
$\beta_r^*k_C.\beta_l^*k_B.\beta_c^*\mu$, and a computation gives
that there is an index set $L=(L_{\ff'},L_{\rb'},L_{\lb'},L_{\rm
lf},L_{\rm rf},L_{\rm cf})$ such that
\begin{align*}
\lefteqn{\beta_r^*k_C.\beta_l^*k_B.\beta_c^*\mu\in
\mc{A}_{\rm phg}^{L}(\bbar{X}\x_0M\x_0\bbar{X};\Omega_b),}\\
&L_{\ff'}=F_{\ff}+ G_{\ff}+\ndemi,&& L_{\rb'}=F_{\rb},&& L_{\lb'}=G_{\lb},\\
&L_{\rm lf}=F_{\ff}+\ndemi,&& L_{\rm rf}=G_{\ff}+\ndemi, &&
L_{\rm cf}=F_{\rb}+G_{\rb}+\ndemi
\end{align*}
and by pushing forward through $\beta_c$ using Melrose \cite[Th. 5]{Me}, we deduce that
the result is polyhomogeneous conormal on $\bbar{X}\x_0\bbar{X}$ with the desired index set.
\end{proof}

In order to analyze the composition $K^*K$ in Subsection \ref{caldproj}, we use the symbolic approach since
it is a slightly more precise (in terms of log terms at the diagonal) than the push-forward Theorem in this case,
and a bit easier to compute the principal symbol of the composition.
We are led to study the composition between  classical operators $K$ and $L$  where $K:C^\infty(\bbar{X}) \to C^\infty(\pl\bbar{X})$ 
is an operator in $I^{-1}(\bbar{X}\x\pl\bbar{X})$ and $L:C^\infty(\pl\bbar{X}) \to C^\infty(\bbar{X})$
is in $I^{-1}(\pl\bbar{X}\x\bbar{X})$. 
We show 
\begin{lemma}\label{compositionKL}
Let $K\in I^{-1}(\bbar{X}\x\pl\bbar{X})$ and $L\in I^{-1}(\pl\bbar{X}\x\bbar{X})$ with principal symbols 
$\sigma_K(y;\xi,\mu)$ and $\sigma_L(y;\xi,\mu)$.
The composition $L\circ K$ is a classical pseudodifferential operator on $\pl \bbar{X}$ in the class
$L\circ K\in \Psi^{-1}(\pl \bbar{X})$. Moreover the principal symbol of $LK$ is given by 
\begin{equation}\label{calculsymbpr}
\sigma_{\rm pr}(L\circ K)(y;\mu)=(2\pi)^{-2}\int_{0}^\infty\hat{\sigma}_{L}(y;-x,\mu). \hat{\sigma}_{K}(y;x,\mu)
dx.
\end{equation}
where $\hat\sigma$ denotes the Fourier transform of $\sigma$ in the variable $\xi$.
\end{lemma}
\begin{proof} Since the composition with smoothing operators is easier, we essentially need to understand the composition of singular kernels like \eqref{operatorsKL}. 
Writing the kernel of $K$ and $L$ as a sum  
of elements $K_j,L_j$ of the form 
\eqref{operatorsKL}, we are reduced to analyze in a chart $U$ 
\[L_jK_j f(y)= \frac{1}{(2\pi)^{2n+2}}\int e^{ix'(\xi'-\xi)+iy'(\mu'-\mu)+iy\mu-iy''\mu'}b(y;\xi', \mu)\chi(x',y') a(y'';\xi, \mu') f(y'')dy''d\Omega \]
where $d\Omega:=dy'dx'd\xi d\xi'd\mu d\mu'$, 
 $\chi\in C_0^\infty(U)$ and $a,b$ are compactly supported in $U$ in the $y$ and $y''$ coordinates.
 If $U$ intersects the boundary $\pl \bbar{X}$, then $\chi$ is supported in $x'\geq 0$.
The kernel of the composition $L_jK_j$ in the chart $U$ is then
\[\begin{split}
F(y,y'')=&\, \frac{1}{(2\pi)^{2n+2}}\int e^{ix'(\xi'-\xi)+iy'(\mu'-\mu)+iy\mu-iy''\mu'}b(y;\xi', \mu)\chi(x',y') a(y'';\xi, \mu') d\Omega \\
=&\, \frac{1}{(2\pi)^{n}}\int e^{i\mu(y-y'')} c(y,y'';\mu)d\mu
\end{split}\]
where 
\[c(y,y'';\mu):=\frac{1}{(2\pi)^{n+2}}\int e^{-iy''.\mu'}b(y;\xi',\mu)a(y''; \xi, \mu-\mu')\hat{\chi}(\xi-\xi',\mu') d\mu'd\xi d\xi'.\]
We want to prove that $c(y,y'';\mu)$ is a symbol of order $-1$ with an expansion in homogeneous terms in 
$\mu$ as $\mu\to \infty$. We shall only consider the case where $U\cap \pl\bbar{X}\not=\emptyset$ since the other 
case is simpler. First, remark that in $U$ the function $\chi$ can be taken of the form 
$\chi(x,y)=\varphi(x)\psi(y)$ with $\psi\in C_0^\infty(\rr^n)$ and $\varphi\in C_0^\infty([0,1))$ equal to $1$ in $[0,1/2]$,
therefore $\hat{\chi}(\xi,\mu)=\hat{\varphi}(\xi)\hat{\psi}(\mu)$ with $\hat{\psi}$ Schwartz and by integration by parts one also  has
\[\hat{\varphi}(\xi)= \frac{1}{i\xi}( 1+ \hat{\varphi'}(\xi)).\]  
with $\hat{\varphi'}$ Schwartz.
We first claim that $|\pl_y^\alpha\pl^\beta_{y''}\pl_\mu^\gamma 
c(y,y'';\mu)|\leq C \cjg \mu\cjd^{-1-|\gamma|}$ uniformly in $y,y''$: indeed using the properties of $\hat{\chi}$ and the symbolic assumptions on $a,b$, we have that for any $N\gg |\beta|$, there is a constant 
$C>0$ such that 
\begin{align*}
\lefteqn{|\pl_y^\alpha\pl^\beta_{y''}\pl_\mu^\gamma 
c(y,y'';\mu)| }\\
&&\leq C\int \Big(\frac{1}{1+|\xi|'+|\mu|}\Big)^{1+k}\Big(\frac{1}{1+|\xi|+|\mu|}\Big)^{1+j}\cjg \xi-\xi'\cjd^{-1}\cjg \mu'\cjd^{-N+|\beta|} d\mu'd\xi d\xi'
\end{align*}
where $j+k=|\gamma|$. Using polar coordinates $i\xi+\xi'=re^{i\theta}$ in $\cc\simeq \rr^2$, the integral above is bounded by  
\[C\int \Big(\frac{1}{1+r|\cos(\theta)|+|\mu|}\Big)^{1+k}\Big(\frac{1}{1+r|\sin\theta|+|\mu|}\Big)^{1+j}\frac{1}{1+r|\cos\theta-\sin\theta|}r drd\theta\]
which, by a change of variable $r\to r|\mu|$ and splitting the $\theta$ integral in different regions, is easily shown to be 
bounded by $C\cjg \mu \cjd^{-1-|\beta|}$.

To prove that $LK$ it is a classical operator of order $-1$ (with an expansion in homogeneous terms), 
we can modify slightly the usual proof of composition of pseudo-differential operators, like in Theorem 3.4 of \cite{GrSj}.
Let $\theta\in C_0^\infty(\rr)$ be an even function equal to $1$ near $0$. 
We write for $\mu=\la \omega$ with $\omega\in S^{n-1}$
\[\begin{split}
F(y,y'')=&\, \frac{1}{(2\pi)^{2n+2}}\int e^{ix'(\xi'-\xi)-i(\mu'-\mu)(y''-y')}\chi(x',y')b(y;\xi',\mu)a(y'';\xi,\mu')dy'dx'd\xi d\mu d\mu' d\xi'\\
 =&\, \frac{1}{(2\pi)^{n}}\int e^{i(y-y'')\mu}c(y,y'';\mu) d\mu
 \end{split}
 \]
with 
\[\begin{split}
\lefteqn{c(y,y';\mu)}\\
&= \frac{\la^{n+1}}{(2\pi)^{n+2}}\int e^{-i\la x'\zeta-i\la \sigma. s}\chi(x',y''-s)b(y;\xi',\la\omega)a(y'';\xi'+\la\zeta,\la(\omega+\sigma))
d\Omega d\xi' \\
&= \frac{\la^{n+1}}{(2\pi)^{n+2}}\int e^{-i\la x'\zeta-i\la \sigma. s}\varphi(x')\theta(\zeta)
\psi(y''-s)b(y;\xi',\la\omega)a(y'';\xi'+\la\zeta,\la(\omega+\sigma))d\Omega d\xi' \\
 & \quad +  \frac{\la^{n+1}}{(2\pi)^{n+2}}\int e^{-i\la x'\zeta-i\la \sigma. s}\varphi(x')(1-\theta)(\zeta)
\psi(y''-s)\\
&\qquad \qquad \qquad \quad \times b(y;\xi',\la\omega)a(y'';\xi'+\la\zeta,\la(\omega+\sigma))d\Omega d\xi' \\
& =:c_1(y,y'';\mu)+c_2(y,y'';\mu)
\end{split}\]
where $\Omega=(\sigma,s,\zeta,x')$. Let us denote the phase by $\Phi:= x'\zeta+\sigma.s$. 
The last integral can be dealt with by integrating by parts in $x'$:
\begin{equation}\label{c_2}
\begin{split}
&{c_2(y,y'';\mu)}\\ 
&=\frac{\la^{n}}{i(2\pi)^{n+2}}\int e^{-i\la\Phi}\varphi'(x')\frac{(1-\theta)(\zeta)}{\zeta}
 \psi(y''-s)b(y;\xi',\la\omega)a(y'';\xi'+\la\zeta,\la(\omega+\sigma))d\Omega d\xi'\\
& + \frac{\la^{n}}{i(2\pi)^{n+2}}\int e^{-i\la \sigma.s}\frac{(1-\theta)(\zeta)}{\zeta}
 \psi(y''-s)b(y;\xi',\la\omega)a(y'';\xi'+\la\zeta,\la(\omega+\sigma))d\sigma dsd\zeta d\xi'.
\end{split}
\end{equation}
We can extend $\varphi'=\pl_x\varphi$ by $0$ on $(-\infty,0]$ to obtain a $C_0^\infty(\rr)$ 
function which vanishes near $0$. Since $\varphi'$ now vanishes near $0$, one easily proves that the first integral in \eqref{c_2} is a $O(\la^{-N})$ for all $N$, uniformly in $y,y''$ by
using integrations by parts $N$ times in $x'$ and $\pl_{x'}(e^{-i\la x'\zeta})=-i\la \zeta e^{-i\la x'\zeta}$. 
Now for the second integral in \eqref{c_2}, we use stationary phase in $(\sigma,s)$, 
one has for any $N\in \nn$
\begin{equation}\label{statphase}
\begin{split}
\lefteqn{\int e^{-i\la \sigma. s}\psi(y''-s)\frac{(1-\theta)(\zeta)}{\zeta}b(y;\xi',\la\omega)a(y'';\xi'+\la\zeta,\la(\omega+\sigma))d\sigma ds}\\
&= (2\pi)^n\frac{(1-\theta)(\zeta)}{\zeta}(\sum_{|\alpha|\leq N} \frac{i^{|\alpha|}}{\alpha!}\pl^\alpha\psi(y'')b(y;\xi',\mu)\pl^\alpha_\mu a(y'';\xi'+\la\zeta,\mu) +S_N(y,y'';\xi',\zeta,\mu))
\end{split}
\end{equation}
with  $|S_N(y,y'';\xi',\zeta,\mu)|\leq C \cjg (\xi',\mu)\cjd^{-1} \cjg (\xi'+|\mu|\zeta,\mu)\cjd^{-1-N}$.
Now, both $a$ and $b$ can be written under the form
$a=a_N+a_h$ and $b=b_h+b_N$ where $a_N(y;\xi,\mu),b_N(y;\xi,\mu)$ 
are bounded in norm by $C\cjg (\xi,\mu)\cjd^{-N}$ and $a_h(y;\xi,\mu),b_h(y,\xi,\mu)$ are finite sums of homogeneous
functions $a_h^{-j},b_h^{-j}$ of order $-j$ in $|(\xi,\mu)|>1$ for $j=1,\dots N-1$. Replacing $a,b$ in \eqref{statphase} by their 
decomposition $a_N+a_h$ and $b_N+b_h$ we get that $c(y,y'',\mu)$ is the sum of a 
term bounded uniformly by $C\cjg \mu\cjd^{-N+2}$ and some terms of the form 
\[\la \int  \frac{1-\theta(\zeta)}{\zeta}b_h^{-j}(y;\xi',\mu)\pl^\alpha \psi(y'')\pl_\mu^\alpha a_h^{-k}(y'';\xi'+\la\zeta,\mu)d\zeta d\xi'.\]
The integral is well defined and is easily seen (by changing variable $\xi'\to \la\xi'$) to be homogeneous  of order $-k-j-|\alpha|+1$
for $\la=|\mu|>1$ . This shows that $c_2(y,y'';\mu)$ has an expansion in homogeneous terms.
It remains to deal with $c_1$. We first apply stationary phase in the $(\sigma,s)$ variables and we get 
\begin{align*}
c_1(y,y'';\mu)= &\frac{\la}{(2\pi)^2}\sum_{|\alpha|\leq N} \frac{i^{|\alpha|}}{\alpha!}\pl^\alpha \psi(y'')\int e^{-i\la x'\zeta}b(y,\xi',\mu)
\varphi(x')\theta(\zeta)\pl^\alpha_\mu a(y'';\xi'+\la\zeta,\mu)dx'd\xi' d\zeta \\
&+\int \varphi(x')S_N'(y,y'';\xi',\zeta,\mu)d\xi' d\zeta dx'
\end{align*}
for some $S_N'$ which will contribute $O(\la^{-N-2})$ like for $c_2$ above. Decomposing $a(y,\xi,\mu)$ and $b(y,\xi,\mu)$ as above in homogeneous terms outside a compact set in $(\xi,\mu)$, it is easy to see that up to a $O(\la^{-N})$ term, we can reduce the 
analysis of $c_1(y,y'';\mu)$ to the case where $a,b$ are replaced by terms $a_h^{-j},b_h^{-k}$ homogeneous
of orders $-j,-k$ outside compacts. We then have 
\begin{equation}\label{homoexp}
\begin{split}
\lefteqn{\int e^{-i\la x'\zeta}b_h^{-j}(y,\xi',\mu)
\varphi(x')\theta(\zeta)\pl^\alpha_\mu a_h^{-k}(y'';\xi'+\la\zeta,\mu)dx'd\xi' d\zeta}\\
&=\la^{-j-k-|\alpha|+1}\int e^{-i\la x'\zeta}b_h^{-j}(y,\xi',\omega)
\varphi(x')\theta(\zeta)\pl^\alpha_\mu a_h^{-k}(y'';\xi'+\zeta,\omega)dx'd\xi' d\zeta
\end{split}
\end{equation}
and we write by Taylor expansion at $\zeta=0$
\begin{equation}\label{taylorexp}
\theta(\zeta)\pl_\mu^\alpha a_h^{-k}(y'';\xi'+\zeta,\omega)=\theta(\zeta)\pl_\mu^\alpha a_h^{-k}(y'';\xi',\omega)+
\zeta \theta(\zeta) a'(y'',\xi',\zeta,\omega)
\end{equation}
for some $a'(y'';\xi',\zeta,\mu)$ smooth in $y''$ and homogeneous of degree $-k-1$ in $|(\xi,\zeta,\mu)|>1$.  
For the term with $a'$, we have by integration by parts in $x'$
\begin{equation}\label{a'}
\begin{split}
\lefteqn{\int \zeta e^{-i\la x'\zeta}b_h^{-j}(y,\xi',\omega)
\varphi(x')\theta(\zeta) \pl^\alpha_\mu  a'(y'';\xi',\zeta,\omega)dx'd\xi' d\zeta}\\
&=(i\la)^{-1}\int e^{-i\la x'\zeta}\varphi'(x')b_h^{-j}(y,\xi',\omega)\theta(\zeta)
\pl^\alpha_\mu a'(y'';\xi',\zeta,\omega)dx'd\xi' d\zeta\\ 
&\quad +(i\la)^{-1}\int b_h^{-j}(y,\xi',\omega)\theta(\zeta)\pl^\alpha_\mu 
a'(y'';\xi',\zeta,\omega)d\xi' d\zeta
\end{split}\end{equation}
and the first term is $O(\la^{-\infty})$ by non-stationary phase while the second one is homogeneous of order $-1$ in $\la$ (the integrals in all variables are converging). It remains to deal with the first term in \eqref{taylorexp}, we notice that $\theta$ is even and so 
\[\int \varphi(x') e^{-i\la x'\zeta}\theta(\zeta)dx'd\zeta=\la^{-1}\int \hat{\theta}(x')\varphi(x'/\la)dx'
=\la^{-1}\pi-\int \hat{\theta}(x')(1-\varphi(x'/\la))dx'.\]
Since $\hat{\theta}$ is Schwartz, the last line clearly has an expansion 
of the form $\pi\la^{-1}+O(\la^{-\infty})$ for some constant $C$, and combining with \eqref{a'}, we deduce that  \eqref{homoexp}  is thus homogeneous of degree 
$\la^{-j-k-1}$ modulo $O(\la^{\infty})$. This ends the proof of the fact that $KL$ is a classical pseudo-differential operator on $M$. 

Now, we compute the principal symbol. According to the discussion above, it is given by 
\begin{align*}
&-i(2\pi)^{-2}\int  \frac{1}{\zeta}\Big(\sigma_{L}(y;\xi',\mu)((1-\theta(\zeta))\sigma_{K}(y;\xi'+\zeta,\mu)+
 \theta(\zeta)\zeta\sigma_K'(y;\xi',\zeta,\mu)\Big)d\xi' d\zeta\\
&+(2\pi)^{-2}\pi \int \sigma_L(y,\xi',\mu) \sigma_K(y;\xi',\mu)d\xi'
\end{align*}
where $\zeta\sigma_K'(y;\xi',\zeta,\mu):=\sigma_K(y;\xi'+\zeta,\mu)-\sigma_K(y;\xi',\mu)$.
It is straightforward to see that this is equal to \eqref{calculsymbpr} by using the fact that the Fourier transform of 
the Heaviside function is the distribution $\pi\delta-i/\zeta$. Notice that the integral \eqref{calculsymbpr} makes sense 
since $\sigma_K,\sigma_L$ are $L^2$ in the $\xi'$ variable.
\end{proof}

\end{document}